\documentclass[11pt]{amsart}
\usepackage{amsmath,amssymb,amsthm, hyperref, amscd}
\usepackage{tikz-cd}
\usepackage{enumitem}
\usepackage[T1]{fontenc}
\usepackage{textcomp}
\usepackage{tikz}
\usetikzlibrary{matrix,arrows,decorations.pathmorphing}
 \usepackage{tikz-cd}
\usepackage[scale=0.95,ttdefault]{AnonymousPro}
\hypersetup{
   colorlinks=true,
   linkcolor=blue,
    urlcolor=cyan,
    }

\usepackage[
    backend=biber,
    style=alphabetic,
    useprefix=true,
    maxnames=100,
    maxcitenames=100,
    sorting=anyt
]{biblatex}
\newcommand{\fa}{\mathfrak{a}}
\newcommand{\fb}{\mathfrak{b}}
\newcommand{\fc}{\mathfrak{c}}

\newcommand{\fm}{\mathfrak{m}}

\newcommand{\fM}{\mathfrak{M}}
\newcommand{\fN}{\mathfrak{N}}
\newcommand{\fp}{\mathfrak{p}}
\newcommand{\fq}{\mathfrak{q}}
\newcommand{\fP}{\mathfrak{P}}
\newcommand{\fQ}{\mathfrak{Q}}

\newcommand{\Spec}{\operatorname{Spec}}
\newcommand{\Supp}{\operatorname{Supp}}
\newcommand{\Ass}{\operatorname{Ass}}
\newcommand{\Ann}{\operatorname{Ann}}

\newcommand{\Fit}{\operatorname{Fit}}

\newcommand{\HT}{\operatorname{ht}}
\newcommand{\rank}{\operatorname{rank}}
\newcommand{\depth}{\operatorname{depth}}

\newcommand{\colim}{\operatorname{colim}}

\newcommand{\Min}
{\operatorname{Min}}
\newcommand{\Max}
{\operatorname{Max}}

\newcommand{\cO}{\mathcal{O}}

\newcommand{\cC}{\mathcal{C}}

\newcommand{\cA}{\mathcal{A}}
\newcommand{\cR}{\mathcal{R}}

\newcommand{\cM}{\mathcal{M}}

\newcommand{\cF}{\mathcal{F}}

\newcommand{\cJ}{\mathcal{J}}

\newcommand{\bP}{\mathbf{P}}

\newcommand{\bR}{\mathbf{R}}
\newcommand{\bQ}{\mathbf{Q}}

\newcommand{\bZ}{\mathbf{Z}}

\newcommand{\cD}{\mathcal{D}}

\newcommand{\Tag}[1]{\href{https://stacks.math.columbia.edu/tag/#1}{\texttt{#1}}}
\newcommand{\citestacks}[1]{\cite[Tag \Tag{#1}]{stacks}}
\newcommand{\citetwostacks}[2]{\cite[Tags \Tag{#1} and \Tag{#2}]{stacks}}

\newcommand{\OS}[1]{\operatorname{OS}_2^\wedge({#1})}

\newtheorem{Thm}{Theorem}[section]

\newtheorem{Lem}[Thm]{Lemma}
\newtheorem{Cor}[Thm]{Corollary}

\theoremstyle{definition}
\newtheorem{Def}[Thm]{Definition}

\newtheorem{Discu}[Thm]{Discussion}
\newtheorem{Exam}[Thm]{Example}

\theoremstyle{remark}

\newtheorem{Rem}[Thm]{Remark}

\newtheorem{Ques}[Thm]{Question}

\title{$(S_2)$-ifications, semi-Nagata rings, and the lifting problem}
\author{Shiji Lyu}
\address{Department of Mathematics, Statistics, and Computer Science\\University of Illinois at Chicago\\Chicago, IL
60607-7045\\USA}
\email{\href{mailto:slyu@uic.edu}{slyu@uic.edu}}
\urladdr{\url{https://homepages.math.uic.edu/~slyu/}}

\begin{document}

\begin{abstract}
This is a two-part article.
In the first part, we study an alternative notion to Nagata rings.
    A Nagata ring is a Noetherian ring $R$ such that every finite $R$-algebra that is an integral domain
    has finite normalization.
    We replace the normalization by an $(S_2)$-ification,
    study new phenomena, and prove parallel results. 
    In particular, we show a Nagata domain has a finite $(S_2)$-ification.
In the second part, we study the local lifting problem.
We show that for a semilocal Noetherian ring $R$ that is $I$-adically complete for an ideal $I$, if $R/I$ has $(S_k)$ (resp. Cohen--Macaulay, Gorenstein, lci) formal fibers, so does $R$.
As a consequence, we show if $R/I$ is a quotient of a Cohen--Macaulay ring, so is $R$.
We also discuss difficulties in lifting geometrically $(R_k)$ formal fibers.
\end{abstract}

\maketitle

For a ring $R$, we write 
$\Min(R)$ for the set of minimal primes of $R$,
$\Spec_1(R)$  the set of primes of height $1$ of $R$,
$\Max(R)$  the set of maximal ideals of $R$.
We also write
$R^\circ=R\setminus\cup\Min(R)$.
In particular, if $R$ is a reduced ring,
then $R^\circ$ is the set of nonzerodivisors in $R$ \citestacks{00EW}.


For an integral domain $R$, $R^\nu$ denotes the normalization of $R$.
The notations $R^{n\sigma}$ and $R^{\sigma}$ are introduced in \S\ref{sec:NaiveandCanS2clos}.

For a scheme $X$, $\cO(X)$ denotes the section ring $\Gamma(X,\cO_X)$.

For a ring $R$ and an ideal $I$ of $R$,
$V(I)$ and $D(I)$ denotes the closed subscheme of $\Spec(R)$ defined by $I$ and its complement.
When $I=fR$ is principal we write $V(f)$ and $D(f)$.

For a ring $R$ and an ideal $I$ of $R$,
a minimal prime divisor of $I$ is an element of $V(I)$ minimal with respect to inclusion.
When $R$ is Noetherian, a prime divisor of $I$ is an element of $\Ass_R(R/I)$.

For a semilocal Noetherian ring $R$, $R^\wedge$ denotes its adic completion with respect to its Jacobson radical.
$R^\wedge$ is a finite product of Noetherian complete local rings.

\section{Introduction}
This is a two-part article motivated by the following classical question,
generally referred to as the lifting problem.
\begin{Ques}[cf. {\cite[Remarque 7.4.8]{EGA4_2}}]\label{ques:formallift}
Let $\bR$ be a property of Noetherian rings.
    Let $R$ be a Noetherian ring, $I$ an ideal of $R$.
    Assume $R$ is $I$-adically complete and $R/I$ is $\bR$.
    Is $R$ always $\bR$?
\end{Ques}
When Question \ref{ques:formallift} admits an affirmative answer, we say $\bR$ has the lifting property.
There have been numerous studies on the lifting problem and its variants,
for many important properties $\bR$.
For example, lifting holds for $\bR$=``Nagata'' \cite{Marot-Nagata-Lift} and $\bR$=``quasi-excellent'' \cite{formal-lifting-excellence-Gabber},  but not for $\bR$=``excellent'' or $\bR$=``universally catenary'' \cite{Greco-Universal-Catenary-No-Lift}.
We refer the reader to \cite[Appendix]{formal-lifting-excellence-Gabber} for more information.

In previous work (see \cite[\S 8]{Lyu-dual-complex-lift}), the author showed that $\bR$=``is a quotient of a Gorenstein ring'' satisfies the lifting property.
However, whether or not $\bR$=``is a quotient of a Cohen--Macaulay ring'' satisfies the lifting property
seems to be difficult.
We provide two perspectives on this question.\\

The first part \S\S\ref{sec:FiniteInclusion}--\ref{sec:liftSagataUC}
discusses a new notion, which the author calls \emph{semi-Nagata rings}.
A Nagata ring is a Noetherian ring $R$ so that for every finite $R$-algebra $B$ that is an integral domain,
$B^\nu$ is finite over $B$.
This is clearly equivalent to the standard definition \citestacks{032R}.
Lifting of the Nagata property is the starting point for lifting of other properties such as quasi-excellence.

We call a ring $R$ \emph{semi-Nagata} if $R$ is Noetherian and for every finite $R$-algebra $B$ that is an integral domain, $B$ admits a finite $(S_2)$-ification,
in the sense that there is a finite inclusion of integral domains $B\to C$ so that $C$ is $(S_2)$ and $B_\fp=C_\fp$ for all $\fp\in\Spec_1(B)$.
We have the following main result, 
which the author believes to be new
(Definition \ref{def:semi-Nagata} and Theorems \ref{thm:semi-Nagatacharacterizelocal} and \ref{thm:semi-Nagatacharacterize})
\begin{Thm}\label{thm:semi-Nagatamain}
Let $R$ be a Noetherian ring.
    \begin{enumerate}[label=$(\roman*)$]
        \item\label{SagataMain:S1fiber} If $R$ is semilocal, then $R$ is semi-Nagata if and only if $R$ has $(S_1)$ formal fibers.
        \item If $R$ is semi-Nagata, then every essentially finitely generated $R$-algebra is semi-Nagata.
        \item $R$ is semi-Nagata if and only if $R$ has $(S_1)$ formal fibers
        and for every $\fp\in\Spec(R)$
        there exists $f\in R,f\not\in\fp$
        so that $(R/\fp)_f$ is $(S_2)$.
        \item $R$ is semi-Nagata if and only if  for every finite $R$-algebra $B$ that is an integral domain,  there is a finite inclusion of integral domains $B\to C$ so that $C$ is $(S_2)$.
    \end{enumerate}
\end{Thm}
The corresponding classical result for Nagata rings is 
\begin{Thm}
Let $R$ be a Noetherian ring.
    \begin{enumerate}[label=$(\roman*)$]
        \item\label{nagata1} If $R$ is semilocal, then $R$ is Nagata if and only if $R$ has geometrically reduced formal fibers.
        \item\label{nagata2} If $R$ is Nagata, then every essentially finitely generated $R$-algebra is Nagata.
        \item\label{nagata3} $R$ is Nagata if and only if $R$ has geometrically reduced formal fibers
        and for every finite $R$-algebra $B$ that is an integral domain 
        there exists $f\in B^\circ$
        so that $B_f$ is normal.
        \item\label{nagata4} $R$ is Nagata if and only if  for every finite $R$-algebra $B$ that is an integral domain,  there is a finite inclusion of integral domains $B\to C$ so that $C$ is normal.
    \end{enumerate}
\end{Thm}

See \cite[\S 7.6 and \S 7.7]{EGA4_2} for \ref{nagata1}\ref{nagata2}\ref{nagata3},
whereas \ref{nagata4} is trivial.
In particular, 
from either \ref{nagata3} or \ref{nagata4} of both theorems,
we have the following result, which the author also believes to be new.
\begin{Cor}\label{cor:Nagataissemi-NagataIntro}
    A Nagata ring is semi-Nagata. 
    In particular, a Nagata domain has a finite $(S_2)$-ification.
\end{Cor}

There are DVRs that are not Nagata,
\cite[Appendix A1, Example 3]{Nagata-local}.
On the other hand we have (Remark \ref{Rem:1dimsemi-Nagata} and Corollary \ref{cor:CMissemi-Nagata})
\begin{Thm}
    A one-dimesional Noetherian ring is semi-Nagata.
    A Cohen--Macaulay ring is semi-Nagata.
\end{Thm}
This is not new, see \cite[\S\S1--2]{Macaulay-Cesnavi}.\\

One might expect that for a Noetherian domain $R$, the ring $R^{n\sigma}:=\bigcap_{\fp\in\Spec_1(R)} R_\fp$ is integral over $R$ and $(S_2)$,
and is the only $(S_2)$-ification of $R$.
This is not the case, even for a quasi-excellent $R$ (Example \ref{exam:LechFONSI}).
This phenomenon is reflected by the obstructions as in Definition \ref{def:FONSI} (``FONSIs'').
When no FONSIs exist,
the expectation is met (Theorem \ref{thm:NS2ofFONSIfree}).
When they do exist, $R^{n\sigma}$ is not  integral over $R$,
and we replace $R^{n\sigma}$ by $R^\sigma:=R^{n\sigma}\cap R^\nu$.

To show $R^\sigma$ is $(S_2)$,
and to show our main Theorem \ref{thm:semi-Nagatamain},
we show that for a semilocal $R$,
there exists a finite subalgebra of $R^\sigma$ that do not have FONSIs (Theorem \ref{thm:killFONSI}).
After all, FONSIs are pretty rare (Remarks \ref{rem:UCnoFONSI} and \ref{Rem:S2noFONSI} and Lemma \ref{lem:FONSIfinite}).
The idea is inspired by a classical argument of Ratliff \cite[\S 31, Lemma 4]{Mat-CRT}.
We use local cohomology to make a conceptual argument.

We warn the reader that for a semi-Nagata ring $R$,
$R^\sigma$ may not be finite (Example \ref{exam:semi-NagataRsigmaNotFinite}),
resulting in an infinite ascending chain of $(S_2)$-ifications;
and $(S_2)$-ifications of modules may not exist (Remark \ref{rem:NoModuleS2ify}).
Again, expectations are met when FONSIs are not present (Corollary \ref{cor:NOFONSISagataNS2finite} and Theorem \ref{thm:naiveS2ofmodulesfiniteforSagataNOFONSI}).

To conclude the first part, we show that lifting of the semi-Nagata property holds for universally catenary rings.
\begin{Thm}[=Theorem \ref{thm:liftSagataUC}]
    Let $R$ be a Noetherian ring, $I$ an ideal of $R$.
    Assume that 
    \begin{enumerate}
        \item $R$ is $I$-adically complete.
        \item\label{liftSagata:quotSagata} $R/I$ is semi-Nagata.
        \item\label{liftSagata:UC}
        $R$ is universally catenary.
    \end{enumerate}
    Then $R$ is semi-Nagata.
\end{Thm}

The author was not able to show lifting in full generality.
However, when restricted to semilocal rings,
much more advances in the lifting problem are made in the second part (\S\S\ref{sec:Local-Lifting-Nishimura}--\ref{sec:Local-lifting-P0-CM1}) of this article.
We repeat the problem in this setting.
\begin{Ques}[local lifting problem]\label{ques:localformallift}
Let $\bR$ be a property of Noetherian rings.
    Let $R$ be a semilocal Noetherian ring, $I$ an ideal of $R$.
    Assume $R$ is $I$-adically complete and $R/I$ is $\bR$.
    Is $R$ always $\bR$?
\end{Ques}
We show (\S\ref{sec:Local-lifting-result})
\begin{Thm}\label{thm:local-lifting-main}
    Question \ref{ques:localformallift} admits an affirmative answer when $\bR$=``has $(S_k)$ formal fibers,'' where $k\geq 0$ is arbitrary,
    ``has Cohen--Macaulay formal fibers,''
    ``has Gorenstein formal fibers,''
    ``has lci formal fibers,''
    and ``is a quotient of a Cohen--Macaulay ring.''
\end{Thm}
Note that when $k=1$, having $(S_k)$ formal fibers is exactly the semi-Nagata property (Theorem \ref{thm:semi-Nagatamain}\ref{SagataMain:S1fiber}).

We show the formal fiber properties in Theorem \ref{thm:local-lifting-main} following the argument of Nishimura \cite{Nishimura-semilocal-lifting}.
The key new input is that we find ideals that define the non-$\bP$-locus (where $\bP$=``$(S_k)$,''``Cohen--Macaulay,'' ``Gorenstein,'' or ``lci'')
of nice rings strictly functorially with respect to nice homomorphisms.
This task was effortless for the properties considered in \cite{Nishimura-semilocal-lifting}.
We explain this in \S\ref{sec:Local-Lifting-Nishimura}.

In our case, such an assignment of ideals can be found with effort. We do a basic reduction in \S\ref{sec:ExtendAssignment},
 saying we just need to assign  $\fm$-primary ideals to complete local rings $(A,\fm)$ that are $\bP$ exactly on the punctured spectrum,
 strictly functorial with respect to flat maps with $\bP$-fibers and $0$-dimensional special fiber.
After the reduction we find a desired assignment for all properties but lci in \S\ref{sec:SkCMGorAssign}.
The case $\bR$=``is a quotient of a Cohen--Macaulay ring'' of Theorem \ref{thm:local-lifting-main} follows from the case $\bR$=``has Cohen--Macaulay formal fibers''
via the argument already present in \cite[\S 8]{Lyu-dual-complex-lift}.

Finding the assignment for the lci property is the most difficult.
This is because the intrinsic invariant that determines a Noetherian ring $A$ is lci or not,
namely the cotangent complex $L_{A/\bZ}$,
does not have finite cohomology modules.
When $A$ is complete local, we can find a regular local ring $R$ and a surjective map $R\to A$,
and $L_{A/R}$ does have finite cohomology modules;
however, this is still fragile with respect to ring maps.
In any case, we need a way to use an ideal to detect the non-flatness of certain non-finite modules,
more explicitly the modules $C_n$ appearing in \cite{Briggs-Iyenger-Cotangent-Complex}.
We investigate their structures in \S\ref{sec:lciAssign}, and successfully define their Fitting invariant (for $n\geq\dim A+2$), which does the trick.
These observations, namely the structure of $C_n$ (Lemma \ref{lem:StructureofCotangentModule}) and Fitting invariant for certain non-finite modules (Definition \ref{def:Fitting}),
may be of their own interest.

Finally, in \S\ref{sec:Local-lifting-P0-CM1}, we treat the lifting of ``$\bP$ in codimension $0$'' and ``Cohen--Macaulay in codimension $1$.''
This section mainly serves as a discussion of difficulties running the arguments in \S\S\ref{sec:Local-Lifting-Nishimura}--\ref{sec:Local-lifting-result} for other standard properties such as $(R_k)$, but some positive results are obtained.
\\

The two parts of this article are not logically dependent on each other.
However, both parts involve $(S_2)$-ifications and equidimensionality.
In the first part this is a focus point,
whereas in the second part this is a safety requirement for the arguments.\\

\textsc{Acknowledgment}. The author thanks Pham Hung Quy for suggesting the author to consider the lifting problem for CM-quotients. The author thanks Linquan Ma, Kevin Tucker, and Wenliang Zhang for helpful discussions. The author was supported by an AMS-Simons Travel Grant.
 
\section{Finite inclusions of Noetherian integral domains}\label{sec:FiniteInclusion}
\begin{Lem}\label{lem:finiteInclMinoverMinAssoverAss}
    Let $R\subseteq R'$ be a finite inclusion of Noetherian semilocal domains.
    Then 
the canonical surjective map $\Spec(R'^\wedge)\to\Spec(R^\wedge)$
restricts to surjective maps
$\Min(R'^\wedge)\to\Min(R^\wedge)$
and
$\Ass(R'^\wedge)\to\Ass(R^\wedge)$.
\end{Lem}
\begin{proof}
    There exists an $f\in R^\circ$ such that $R'_f$ is flat over $R_f$,
    so $(R'^\wedge)_f$ is flat over $(R^\wedge)_f$.
    As $f$ is a nonzerodivisor in both $R^\wedge$ and $R'^\wedge$
    we get the desired result, cf. \citetwostacks{00ON}{0337}.
\end{proof}

\begin{Lem}\label{lem:FactorFlatBir}
    Let $R\subseteq R'$ be a finite inclusion of Noetherian integral domains.
    Then there exists a factorization $R\subseteq R''\subseteq R'$ so that $R\to R''$ is flat and $R''\to R'$ is birational.
\end{Lem}
\begin{proof}
    Let $x'\in R'^\circ$ be not in the fraction field $K$ of $R$.
    Let $f_{x'}(T)=\sum a_i T^i$ be the monic minimal polynomial of $x'$ over 
    $K$, and let $d=\deg f_{x'}$.
    Then for $x\in R^\circ$ the monic minimal polynomial of $xx'$ over 
    $K$ is $f_{xx'}(T)=\sum x^{d-i}a_i T^i.$
    Therefore we may choose $x$ so that $f_{xx'}\in R[T]$,
    so $R[xx']\cong R[T]/f_{xx'}(T)$ is flat over $R$.
    Inductively we can find our $R''$.
\end{proof}

\begin{Lem}[cf. {\cite[(33.11)]{Nagata-local}}]\label{lem:MinToAssNOETH}
    Let $R\subseteq R'$ be a finite inclusion of Noetherian integral domains.
    Let $x\in R^\circ$.
    Then every minimal prime divisor of $xR'$ lies above a prime divisor of $xR$.
\end{Lem}
\begin{proof}
By Lemma \ref{lem:FactorFlatBir} we may assume $R$ and $R'$ has the same fraction field.
Let $\fp'$ be a minimal prime divisor of $xR'$ and let $\fp=\fp'\cap R$.

    Let $b\in R^\circ$ be such that $bR'\subseteq R$ and that $b\in\fp$.
As $\HT(\fp')=1$, $\fp'^nR'_{\fp'}\subseteq bR'_{\fp'}$
for some $n$.
We can therefore find an $s\in R'\setminus\fp'$ such that $s\fp'^n\subseteq bR'\subseteq R$,
and consequently 
$bs^t\fp'^{nt}\subseteq b^{t+1}R'\subseteq b^{t}R$ for all $t$.
If $\fp\not\in\Ass_{R}(R/xR)$,
then $\depth R_\fp\geq 2$,
so $\fp\not\in\Ass_{R}(R/bR)$,
and we can take $c\in\fp$ a nonzerodivisor on $R/bR$,
thus a nonzerodivisor on $R/b^{t}R$ for all $t$.
As $bs^t\fp'^{nt}\subseteq b^{t}R$,
we have $c^{nt}bs^t\in b^{t}R$,
so $bs^t\in b^tR$ as $bs^t\in bR'\subseteq R$.
Then $b\in b^tR'_{\fp'}$ for all $t$,
contradiction.
Therefore $\fp\in\Ass_R(R/xR)$.
\end{proof}

\section{Subalgebras of normalization}
See \cite[\S 33]{Nagata-local} for  relevant materials.
\begin{Thm}\label{thm:Krull}
    Let $R$ be a Noetherian integral domain,
    $S$ a subalgebra of $R^\nu$.
    Then for every $\fp\in\Spec(R)$,
    there are only finitely many $\fq\in\Spec(S)$ above $\fp$,
    and $\kappa(\fq)$ is finite over $\kappa(\fp)$ for all $\fq$.
\end{Thm}
\begin{proof}
    If $S=R^\nu$,
    then this is part of \cite[Theorem 33.10]{Nagata-local}.
    The general case follows from the fact $\Spec(R^\nu)\to \Spec(S)$ is surjective.
\end{proof}

\begin{Lem}[cf. {\cite[(33.11)]{Nagata-local}}]\label{lem:MinToAss}
    Let $R$ be a Noetherian integral domain,
    $S$ a subalgebra of $R^\nu$.
    Let $a\in S^\circ$,
    and let $\fq$ be a minimal prime divisor of $aS$.   
    
    Then the following hold.
    \begin{enumerate}[label=$(\roman*)$]
\item\label{MtoA:ht} There exists a finite subalgebra $R'$ of $S$ such that $\HT(\fq\cap R')=1$.
        \item\label{MtoA:Ass} If $a\in R$, then $\fp:=\fq\cap R$ is a prime divisor of $aR$. 
    \end{enumerate}
\end{Lem}
\begin{proof}
By Theorem \ref{thm:Krull},
we can take
a finite subalgebra $R'\subseteq S$
so that $a\in R'$ and $\fq$ is the only prime ideal of $S$ above $\fp'=\fq\cap R'$.
If $\fp'$ were not a minimal prime divisor of $aR'$,
then we can find primes $\fp_0'\subsetneq\fp'$ in $R'$ containing $a$.
We can then find primes $\fq_0\subsetneq\fq_1$ of $S$ lying above $\fp_0'\subsetneq\fp'$;
they automatically contain $a$.
By uniqueness, $\fq_1=\fq$,
contradicting the minimality of $\fq$.
Therefore $\fp'$ is a minimal prime divisor of $aR'$,
so $\HT(\fp')=1$ as $R'$ is Noetherian.
This is \ref{MtoA:ht}, and we get \ref{MtoA:Ass} by
Lemma \ref{lem:MinToAssNOETH}.
\end{proof}
\begin{Thm}\label{thm:SisKrull}
     Let $R$ be a Noetherian integral domain,
    $S$ a subalgebra of $R^\nu$.
    Then for every $a\in S^\circ$,
    the set of minimal prime divisors $\fq$ of $aS$ is finite,
    and for every $\fq$,
    $S_\fq$ is a $1$-dimensional Noetherian ring.
\end{Thm}
\begin{proof}
    We may assume $a\in R$.
    Finiteness follows from Lemma \ref{lem:MinToAss}\ref{MtoA:Ass}
    and Theorem \ref{thm:Krull}.
    For each $\fq$,
    Lemma \ref{lem:MinToAss}\ref{MtoA:ht} gives a map $R'_{\fq\cap R'}\to S_\fq$,
    so $S_\fq$ is a $1$-dimensional Noetherian ring by the theorem of Krull--Akizuki \cite[Theorem 33.2]{Nagata-local}.
\end{proof}

\begin{Def}
    Let $R$ be a Noetherian integral domain,
    $S$ a subalgebra of $R^\nu$.
    We say $S$ is $(S_2)$
    if for all $a\in S^\circ$,
    $aS$ is a finite intersection of primary ideals of height $1$.
    This is the same as Serre's condition $(S_2)$ if $S$ is Noetherian.
\end{Def}

\begin{Lem}\label{lem:S2assPrincipal}
     Let $R$ be a Noetherian integral domain,
    $S$ a subalgebra of $R^\nu$.
    Then  the following are equivalent.
    \begin{enumerate}[label=$(\roman*)$]
        \item\label{S2ass:S2} $S$ is $(S_2)$.
        \item\label{S2ass:ass} For all $a\in S^\circ$,
    the set of zero divisors on $S/aS$ is the union of minimal prime divisors of $aS$.
        \item\label{S2ass:intersect} $S=\bigcap_{\fq\in\Spec_1(S)}S_\fq$.
    \end{enumerate}
\end{Lem}
\begin{proof}
That \ref{S2ass:S2} implies \ref{S2ass:ass} is clear as minimal prime divisors of and the set of zero divisors modulo $aS$ can be read off of a primary decomposition, cf. \cite[Proposition 4.7]{AMcommalg}.

Assume \ref{S2ass:ass}. Let $z=x/y\in \bigcap_{\fq\in\Spec_1(S)}S_\fq$ where $x,y\in S^\circ$.
Then $(xS:_S yS)$ is not contained in any minimal prime of $yS$,
as they are of height $1$ (Theorem \ref{thm:SisKrull}).
Therefore $(xS:_S yS)$ contains a nonzerodivisor on $S/yS$
    by prime avoidance, so $x\in yS,z\in S$.

    Finally, assume \ref{S2ass:intersect}.
    Then $aS=\bigcap_{\fq\in\Spec_1(S)} aS_\fq$ for all $a\in S^\circ$,
    so  $aS=\bigcap_{\fq\in\Spec_1(S)} (aS_\fq\cap S)$.
    By Theorem \ref{thm:SisKrull}, all but finitely many of the ideals $aS_\fq\cap S$ are $S$, and the others are $\fq$-primary.
    This gives a primary decomposition of $aS$.
\end{proof}

\begin{Lem}\label{lem:S2localization}
    Let $R$ be a Noetherian integral domain, $S$ a subalgebra of $R^\nu$.
    Then the following are true.
    \begin{enumerate}[label=$(\roman*)$]
        \item\label{S2loc:localize} For every multiplicative subset $W$ of $R$, $W^{-1}S$ is $(S_2)$.
        \item\label{S2loc:localcheck} If $S_\fm$ is $(S_2)$ for all $\fm\in\Spec(R)$,
        then $S$ is $(S_2)$.
    \end{enumerate}
\end{Lem}
\begin{proof}
    \ref{S2loc:localize} is trivial as primary decompositions localize,
    whereas \ref{S2loc:localcheck} follows from Lemma \ref{lem:S2assPrincipal}.
\end{proof}

\begin{Lem}\label{lem:S2filteredunion}
    Let $R$ be a Noetherian integral domain, $S$ a subalgebra of $R^\nu$.
    Assume that $S$ is the filtered union of $R$-subalgebras $\{S_\alpha\}_\alpha$,
    and that each $S_\alpha$ is $(S_2)$.
    Then $S$ is $(S_2)$.
\end{Lem}
\begin{proof}
    Let $a\in S^\circ$;
    we may assume $a\in S_\alpha$ for all $\alpha$.
    If $\fp_\alpha$ is a minimal prime divisor of $aS_\alpha$,
    then there exists a minimal prime divisor $\fp$ of $aS$ above $\fp_\alpha$,
    as $\Spec(S)\to\Spec(S_\alpha)$ is surjective.
    Now for a $b\in S$ not in any minimal prime divisor of $aS$,
    which we may assume in $S_\alpha$ for all $\alpha$,
    we have $b$ not in any minimal prime divisor of $aS_\alpha$.
    Therefore $b$ is a nonzerodivisor on $S_\alpha/aS_\alpha$ by Lemma \ref{lem:S2assPrincipal}.
    Consequently, $b$ a nonzerodivisor on $S/aS=\colim_\alpha S_\alpha/aS_\alpha$,
    so $S$ is $(S_2)$ by Lemma \ref{lem:S2assPrincipal}.
\end{proof}

\section{Naive and canonical $(S_2)$-closures}\label{sec:NaiveandCanS2clos}
\begin{Def}\label{def:S2ify}
    Let $S$ be an integral domain.
We write $S^{n\sigma}$ for $\bigcap_{\fp\in\Spec_1(S)} S_\fq$
    and $S^{\sigma}$ for $S^{n\sigma}\cap S^\nu$.

For a Noetherian integral domain $R$ and  a subalgebra $S$ of $R^\nu$,
we know $S$ is $(S_2)$ if and only if $S=S^{n\sigma}$ (Lemma \ref{lem:S2assPrincipal}).
    We will see $S^{\sigma}$ is $(S_2)$ (Theorem \ref{thm:SsisS2}),
    so $S$ is $(S_2)$ if and only if $S=S^{\sigma}$.
\end{Def}
\begin{Lem}\label{lem:NS2localization}
Let $R$ be a Noetherian integral domain, $S$ a subalgebra of $R^\nu$.

Let $W$ be a multiplicative subset of $S$.
Then $W^{-1}(S^{n\sigma})=(W^{-1}S)^{n\sigma}$ and $W^{-1}(S^{\sigma})=(W^{-1}S)^{\sigma}$.
\end{Lem}
\begin{proof}
    We will show $W^{-1}(S^{n\sigma})=(W^{-1}S)^{n\sigma}$;
    the corresponding identity $W^{-1}(S^{\sigma})=(W^{-1}S)^{\sigma}$
    follows as localization commutes with finite intersections.

    The inclusion $W^{-1}(S^{n\sigma})\subseteq (W^{-1}S)^{n\sigma}$ is clear.
    For the other inclusion,
    let $z=x/y\in (W^{-1}S)^{n\sigma}$ where $x,y\in S^\circ$.
    Let $\fq_1,\ldots,\fq_m,\fq_{m+1},\ldots,\fq_n$ be the minimal prime divisors of $yS$ (there are only finitely many, Theorem \ref{thm:SisKrull}),
    ordered  in a way that $\fq_1,\ldots,\fq_m$ are disjoint from $W$ and $\fq_{m+1},\ldots,\fq_n$ are not.
    Then $z\in S_{\fq_j}$ for all $1\leq j\leq m$,
    and $z\in S_{\fq}$ for all $\fq\in\Spec_1(S)\setminus\{\fq_1,\ldots,\fq_m,\fq_{m+1},\ldots,\fq_n\}$.
    As $S_{\fq_j}$ is a $1$-dimensional Noetherian local ring (Theorem \ref{thm:SisKrull}), we can take $w\in W$ so that $w\in yS_{\fq_j}$ for all $m+1\leq j\leq n$.
    Then $wz=wx/y$ is in $S^{n\sigma}$.
\end{proof}

\begin{Lem}\label{lem:reverseIncl}
Let $S$ be an integral domain,
    Let $S'$ be a subalgebra of $S^{n\sigma}$ (resp. $S^{\sigma}$).
    Then $S'^{n\sigma}\subseteq S^{n\sigma}$ (resp. $S'^{\sigma}\subseteq S^{\sigma}$).
\end{Lem}
\begin{proof}
    For every $\fq\in\Spec_1(S)$,
    we have $S\subseteq S'\subseteq S_\fq$, so
    $S'_\fq=S_\fq$.
    Therefore the family of rings defining the intersection of $S'^{n\sigma}$ contains that of $S^{n\sigma}$,
     giving the $(-)^{n\sigma}$ case.

     If $S'\subseteq S^{\sigma}$, then $S'^{\nu}=S^{\nu}$, which gives the $(-)^{\sigma}$ case.
\end{proof}

\begin{Exam}\label{exam:LechFONSI}
    By a theorem of Lech \cite{Lech-completion-domain},
    a complete Noetherian local ring containing a field is the completion of a Noetherian local domain if and only if its depth is at least $1$.
    Therefore, there exist a Noetherian local domain $R$ of dimension $2$ whose completion is $k[[x,y,z]]/(x,y)\cap(z)$,
    where $k$ is a field.

    In this case, $R^{n\sigma}=\cO(U)$,
    where     $U$ is the punctured spectrum of $R$.
    Therefore $R^{n\sigma}\otimes_R R^\wedge=\cO(U^\wedge)$,
    where    $U^\wedge$ is the punctured spectrum of $R^\wedge$.
    By the specific form of $R^\wedge$
    we see $\cO(U^\wedge)=k((z))\times k[[x,y]]$,
    so $\cO(U^\wedge)$ is not integral over $R^\wedge$,
    and $R^{n\sigma}$ is not integral over $R$.

    If the field $k$ has characteristic zero, then we can even make $R$ quasi-excellent;
    in fact, a Noetherian complete local ring containing a field of characteristic zero is the completion of a quasi-excellent local domain if and only if it is reduced. 
    See \cite[Proof of Theorem 9]{Loepp-03-complete-excellent-domains}.
\end{Exam}

\begin{Def}\label{def:FONSI}
    Let $R$ be a semilocal Noetherian domain.
    A \emph{formal obstruction to naive $(S_2)$-ification for $R$}, or \emph{FONSI for $R$}, is a $P\in \Spec(R^\wedge)$
    such that $\HT(P\cap R)>1$ and that there exists $P_0\in\Min(R^\wedge)$ contained in $P$ such that $\HT(P/P_0)=1$.
    Note that in particular $\HT(P)>1$.

    The set of such $P$ is denoted $\OS{R}.$
    There is a canonical identification $\OS{R}=\bigsqcup_{\fm\in\Max(R)}\OS{R_\fm}$.
    For a non-semi-local $R$, we abuse notations and write $\OS{R}$ for $\bigsqcup_{\fm\in\Max(R)}\OS{R_\fm}$.
\end{Def}

\begin{Rem}\label{rem:UCnoFONSI}
    If $R$ is universally catenary,
   then $\OS{R}=\emptyset$.
    To see this, we may assume $R$ is local, so $R^\wedge$ is (catenary and) equidimensional by Ratliff \citestacks{0AW6},
    so $\HT(P)=\HT(P/P_0)$ for all $P_0\in\Min(R^\wedge)$ contained in $P\in\Spec(R^\wedge)$.
\end{Rem}

\begin{Rem}\label{Rem:S2noFONSI}
Assume $R$ is semilocal.
For $P\in\OS{R}$,
the punctured spectrum of the ring $(R^\wedge)_{\fP}$ is disconnected, as it has an isolated point given by a minimal prime $P_{0}\subsetneq P$ with $\HT(P/P_{0})=1$.
By \citestacks{0BLR} we have $\depth (R^\wedge)_{P}<2$, so $\depth R_{P\cap R}<2$.
   Therefore, if $R$ is $(S_2)$,
   then $\OS{R}=\emptyset$.
\end{Rem}

The rest of this section devotes to the study of rings with no FONSIs.

\begin{Thm}\label{thm:NS2ofFONSIfree}
    Let $R$ be a Noetherian integral domain.
    Assume that $\OS{R}=\emptyset$. 
    Then the following are true.
    \begin{enumerate}[label=$(\roman*)$]
        \item\label{NS2:integral} $R^{n\sigma}$ is integral over $R$, so $R^{n\sigma}=R^\sigma$.
        \item\label{NS2:S2} $R^\sigma$ is $(S_2)$.
         \item\label{NS2:S2unique} If a subalgebra $S$ of $R^\sigma$ is $(S_2)$,
         then $S=R^\sigma$.
        \item\label{NS2:formalS1} If $R$ is semilocal and $R^\wedge$ is $(S_1)$,
        then $R^\sigma$ is finite over $R$.
    \end{enumerate}
\end{Thm}
For a converse to \ref{NS2:integral} see Corollary \ref{cor:noFONSIisNS2integral}.
\begin{proof}
    By Lemmas \ref{lem:NS2localization} and \ref{lem:S2localization},
    we may assume $R$ is local.
    Let $K$ be the fraction field of $R$.
    
    Let $0=Q_1\cap\ldots\cap Q_m\cap Q_{m+1}\cap\ldots\cap Q_n$ be a shortest primary decomposition of $0$ in $R^\wedge$,
    ordered in a way that $P_j=\sqrt{Q_j}$ is a minimal prime of $R^\wedge$ for $1\leq j\leq m$ and not for $m+1\leq j\leq n$.
    Let $S_j=R^\wedge/Q_j$ for $1\leq j\leq m$,
    $S=\prod_{j=1}^m S_j$.
    Let $L_j=(S_j)_{P_j}$,
    so $L=\prod_{j=1}^m L_j$ is the total fraction ring of $S$.
    
    Every $a\in R^\circ$ is a nonzerodivisor on $R^\wedge,R^{n\sigma}\otimes_R R^\wedge$,
    and $S$.
    Therefore there is a commutative diagram
\[
\begin{tikzcd}
R^\wedge \arrow[hookrightarrow,r] \arrow[d] &[0.5em] R^{n\sigma}\otimes_R R^\wedge \arrow[hookrightarrow,r] \arrow[d] &[0.5em] K\otimes_R R^\wedge\arrow[d] \\
S \arrow[hookrightarrow,r]                                           & T \arrow[hookrightarrow,r] & L
\end{tikzcd}
\]
where $T=\prod_{j=1}^m T_j$ and each $T_j$ denote the image of $R^{n\sigma}\otimes_R R^\wedge$ in $L_j$.
Note that the kernel of $K\otimes_R R^\wedge\to L$ is $\bigcap_{j=1}^m Q_j(K\otimes_R R^\wedge)$,
which is nilpotent,
and is zero if $R^\wedge$ is $(S_1)$.

Observe that for every $\fp\in\Spec_{1}(R)$,
we have $R_\fp=(R^{n\sigma})_\fp$.
As $\OS{R}=\emptyset$,
we see that for every $P\in\Spec(R^\wedge)$ so that $\HT(P/P_j)=1$ for some $1\leq j\leq m$,
we have $(R^\wedge)_{P}=(R^{n\sigma}\otimes_R R^\wedge)_{P}$.
Therefore the sub-$S_j$-algebra $T_j$ of $L_j$ satisfies $(S_j)_{P}=(T_j)_{P}$ for all $P\in \Spec_1(S_j)$.

\begin{Lem}\label{lem:naiveS2finiteforcomplete}
    Let $A$ be a Noetherian complete local ring that is $(S_1)$ and has irreducible spectrum.
    Let $F$ be the total fraction ring of $A$ and let $M\subseteq N$ be two sub-$A$-modules of $F$. 
    If $M_{P}=N_{P}$ for all $P\in \Spec_1(A)$,
    and $M$ is finite,
    then $N$ is finite.
\end{Lem}
\begin{proof}
We will use the fact a Noetherian complete local ring is excellent, \citestacks{07QW}.

   Let $U$ be the locus where $M$ is $(S_2)$.
   Then $U$ is open \cite[Proposition 6.11.6]{EGA4_2},
   and contains $\Spec_1(A)$ as $A$ is $(S_1)$.
   Let $j:U\to \Spec(A)$ be the canonical open immersion.
  It is clear that $j_*j^*M=\bigcap_{P\in\Spec_1(A)} M_{P}$.

Let $P_0$ be the minimal prime of $A$,
and let $\overline{A}=A/P_0$.
Note that $\Ass_A(M)=\{P_0\}$.
   For $P\in \Spec(A)\setminus U$,
the completion $\overline{A}^\wedge_{P}$ is $(S_1)$ as $\overline{A}$ is $(S_1)$ with $(S_1)$ formal fibers.
Moreover, by Ratliff \citestacks{0AW6}, $\overline{A}^\wedge_{P}$ is equidimensional,
as $\overline{A}_{P}$ is equidimensional and universally catenary.
Therefore for all $\fp\in\Ass(\overline{A}^\wedge_{P})$ we have $\dim \overline{A}^\wedge_{P}/\fp = \HT(P)>1$.
By \citestacks{0BK3},
$j_*j^*M$ is finite;
and so is its submodule $N$.
\end{proof}
By the lemma, each $T_j$ is finite over $S_j$,
so $T$ is finite over $S$.
The kernel of $R^{n\sigma}\otimes_R R^\wedge\to T$ is nilpotent, and is zero if $R^\wedge$ is $(S_1)$.
Therefore $R^{n\sigma}\otimes_R R^\wedge$ is integral over $R^\wedge$,
and is finite over $R^\wedge$ if $R^\wedge$ is $(S_1)$.
As $R\to R^\wedge$ is faithfully flat,
we get \ref{NS2:integral} and \ref{NS2:formalS1}, cf. \citetwostacks{02L9}{02LA}.

Let $S$ be an arbitrary subalgebra of $R^{n\sigma}=R^\sigma$.
Then $S^{n\sigma}=R^{n\sigma}$ by Lemmas \ref{lem:reverseIncl} and \ref{lem:NoFonsiht1Preim} below.
Therefore $S$ is $(S_2)$ if and only if $S=R^{n\sigma}$.
This gives \ref{NS2:S2}\ref{NS2:S2unique}.
\end{proof}

\begin{Lem}\label{lem:NoFonsiht1Preim}
Let $R\subseteq R'$ be a finite inclusion of Noetherian integral domains, $S'$ a subalgebra of $R'^\nu$.
    Then for every $\fq'\in\Spec_1(S')$, we have either $\fq'\cap R\in\Spec_1(R)$,
    or $\fq'\cap R=P\cap R$ for some $P\in \OS{R}$;
    in particular, if $\OS{R}=\emptyset$,
    then $\fq'\cap R\in\Spec_1(R)$, and
    $R^{n\sigma}\subseteq S'^{n\sigma}$.
\end{Lem}
\begin{proof}
The ``in particular'' assertion follows at once as $S'^{n\sigma}=\bigcap_{\fq'\in\Spec_1(S')} S'_{\fq'}\supseteq \bigcap_{\fq'\in\Spec_1(S')} R_{\fq'\cap R}\supseteq R^{n\sigma}$. 

By Lemma \ref{lem:MinToAss}, we may assume $S'=R'$.
We may assume $R$ local, so $R'$ is semilocal.
    By Lemma \ref{lem:finiteInclMinoverMinAssoverAss}
    every minimal prime of $R'^\wedge$ lies above a minimal prime of $R^\wedge$.

    Let    $P'$ be a minimal prime divisor of $\fq'R'^\wedge$, so $\HT(P')=1$ and $P'\cap R'=\fq'$.
    Take $P_0'\in\Min(R'^\wedge)$ contained in $P'$, so $\HT(P'/P_0')=1$.
    As $R^\wedge$ is universally catenary,
    we have $\HT(P'\cap R^\wedge/P_0'\cap R^\wedge)=1$ by the dimension formula \citestacks{02IJ}.
    Therefore either $\HT(P'\cap R)=1$ or $P'\cap R\in\OS{R}$.
    As $P'\cap R=\fq'\cap R$ this proves the lemma.
\end{proof}

\begin{Cor}\label{cor:NS2ofFONSIfreeS}
    Let $R$ be a Noetherian integral domain, $S$ a subalgebra of $R^\nu$.
    Assume that $\OS{R}=\emptyset$. 
    Then the following are true.
    \begin{enumerate}[label=$(\roman*)$]
        \item\label{NS2c:integral} $S^{n\sigma}$ is integral over $S$, so $S^{n\sigma}=S^\sigma$.
        \item\label{NS2c:S2} $S^\sigma$ is $(S_2)$.
         \item\label{NS2c:S2unique} If a subalgebra $S'$ of $S^\sigma$ is $(S_2)$,
         then $S'=S^\sigma$.
    \end{enumerate}
\end{Cor}
\begin{proof}
Let $S'$ be a $S$-subalgebra of $R^\nu$.
Let $\fq'\in \Spec_1(S')$.
Then $\HT(\fq'\cap R)=1$ by Lemma \ref{lem:NoFonsiht1Preim},
so $\HT(\fq'\cap S)=1$.
Therefore $S^{n\sigma}\subseteq S'^{n\sigma}$.
Apply this to $S'=R^\nu$,
noting that $(R^\nu)^{n\sigma}=R^\nu$ as $R^\nu$ is a Krull domain \cite[Theorem 33.10]{Nagata-local}, we see \ref{NS2c:integral} holds.

Now apply the inclusion $S^{n\sigma}\subseteq S'^{n\sigma}$ 
and Lemma \ref{lem:reverseIncl}
to a $S$-subalgebra $S'$ of $S^{n\sigma}=S^\sigma$.
We get $S^{n\sigma}= S'^{n\sigma}$,
giving \ref{NS2c:S2}\ref{NS2c:S2unique}.
\end{proof}

\section{Extension rings via first local cohomology}\label{sec:ExtendbyLocCoh}
This is an auxiliary section that provides a conceptual variant of Ratliff's construction \cite[\S 31, Lemma 4]{Mat-CRT}.
\begin{Discu}\label{discu:DefR+M}
Let $R$ be a Noetherian ring and let $I$ be an ideal of $R$.
Then there exists a canonical exact sequence
\[\begin{tikzcd}
    0\arrow[r] & H^0_I(R)\arrow[r] & R \arrow[r] & \cO(D(I)) \arrow[r] & H^1_I(R)\arrow[r] & 0.
\end{tikzcd}\]
For a submodule $M$ of $H^1_I(R)$,
we denote by $R+_I M$ the unique submodule of $\cO(D(I))$ that contains the image of $R$ and has image $M$ of $H^1_I(R)$.
Denote by $R[I;M]$ the subalgebra  of $\cO(D(I))$ generated by $R+_I M$,
and denote by $M^a$ the  unique submodule of $H^1_I(R)$ such that $R+_I M^a=R[I;M]$.
We thus have a commutative diagram with exact rows
\[\begin{tikzcd}
    0\arrow[r] & H^0_I(R)\arrow[r] \arrow[equal,d] & R \arrow[r] \arrow[equal,d] & R+_I M \arrow[r] \arrow[hookrightarrow,d] & M\arrow[r]\arrow[hookrightarrow,d] & 0\\
    0\arrow[r] & H^0_I(R)\arrow[r] & R \arrow[r] & R[I;M] \arrow[r] & M^a\arrow[r] & 0.
\end{tikzcd}\]
There is an obvious functoriality with respect to ring maps,
giving compatibility with flat base change; and an obvious functoriality with inclusion of submodules. 

The module $R+_I M$ is finite if and only if $M$ is finite; in which case $R[I,M]$ is a finitely generated $R$-algebra.
\end{Discu}

\begin{Discu}\label{discu:Rmoda+M}
   If an ideal $\fa\subseteq R$ is $I^\infty$-torsion,
then for $\overline{R}=R/\fa$ and $\overline{I}=(I+\fa)/\fa$ we have $H^1_I(R)=H^1_{\overline{I}}(\overline{R})$, and $R+_I M=\overline{R} +_{\overline{I}} M$. 
In particular, taking $\fa=H^0_I(R)$, we reduce to the case $H^0_I(R)=0$,
or equivalently, $I$ contains a nonzerodivisor on $R$.
\end{Discu}

\begin{Discu}\label{discu:R+MTwoSteps}
As $R+_IM\subseteq\cO(D(I))$ we have $H^0_I(R+_I M)=0$.
There is a canonical identification $H^1_I(R+_I M)=H^1_I(R)/M$,
as $H^1_I(R)=H^1_I(R/H^0_I(R))$.
For a submodule $N$ containing $M^a$, this gives canonical identifications\\
$R[I;M]+_{IR[I;M]}N/M^a=R+_IN$
and
$R[I;M][IR[I;M];N/M^a]=R[I;N]$.
\end{Discu}

\begin{Discu}\label{discu:R+MdifferentIdeals}
    Assume $H^0_I(R)=0$ and let $J$ be an ideal of $R$ containing $I$.
    Then the canonical identification $R\Gamma_J(R\Gamma_I(R))=R\Gamma_J(R)$
    gives a canonical identification $H^0_J(H^1_I(R))=H^1_J(R)$.
    Therefore, if $M$ is a submodule of $H^1_J(R)$,
    then it can be viewed as  a submodule of $H^1_I(R)$,
    and we have canonical identifications $R+_I M=R+_JM$ and $R[I;M]=R[J;M]$.
\end{Discu}

\begin{Lem}[cf. \citestacks{0BHZ}]\label{lem:onestepExtend}
   Let $(R,\fm)$ be a Noetherian local ring of depth at least $1$ that is not a DVR.
    Let $M$ be a submodule of $H^1_\fm(R)$ that is annihilated by $\fm$.
    Then $R[\fm;M]$ is integral over $R$.
\end{Lem}
\begin{proof}
As $\depth R\geq 1$, $R$ is a subring of $\cO(D(\fm))$. 
    Let $y\in R+_\fm M$. Then $y\fm\subseteq R$.
    If $y\fm=R$, then we can write $1=yt$ for some $t\in \fm$, so $a=ayt\in tR$ for all $a\in\fm$, and $\fm=tR$ is principal,  contradiction.
    Therefore $y\fm=\fm$, so $y$ is integral over $R$ by \citestacks{0B5T} and the fact $\fm$ contains a nonzerodivisor on $R$,
    which is automatically a nonzerodivisor on $\cO(D(\fm))$. 
\end{proof}

\begin{Def}\label{def:conservativeExtensions}
    Let $R$ be a Noetherian ring and let $U$ be a scheme-theoretically dense open subset of $\Spec(R)$.
    A \emph{finite $U$-modification} of $R$ is a finite $R$-subalgebra $R'$ of $\cO(U)$.

    Scheme-theoretically, this means 
    $R\to R'$ is finite, and
    $\Spec(R')\times_{\Spec(R)}U$ is scheme-theoretically dense in $\Spec(R')$ and maps isomorphically onto $U$.
\end{Def}

\begin{Def}\label{def:conservative-inertPrimes}
    Let $R$ be a Noetherian integral domain.
    We say $\fp\in\Spec(R)$ is \emph{p-unibranch} if, for every finite $D(\fp)$-modification $R'$ of $R$, 
    the ring $R'_\fp$ is local.
\end{Def}

We record the following standard fact.
\begin{Lem}\label{lem:extendUmodif}
    Let $R$ be a Noetherian integral domain and let $U$ be an open subset of $\Spec(R)$.
    Let $W$ be a multiplicative subset of $R$ and let $W^{-1}U$ be the preimage of $U$ in $\Spec(W^{-1}R)$.
    Then every finite $(W^{-1}U)$-modification of $W^{-1}R$ is of the form $W^{-1}R'$ where $R'$ is a finite $U$-modification of $R$.
    In particular, if $\fp\in\Spec(R)$ is disjoint from $W$,
    then $\fp\in\Spec(R)$ is p-unibranch
    if and only if $W^{-1}\fp\in\Spec(W^{-1}R)$ is p-unibranch.
\end{Lem}
\begin{proof}
    By standard theory of limits, particularly \citetwostacks{01ZO}{081E},
    we may assume $W$ is generated by a single element $f$.
    Let $T$ be a finite $(W^{-1}U)$-modification of $W^{-1}R$.
    Then $T$ and $\cO_U$ glue to a coherent $\cO_{V}$-algebra $\cA$,
    where $V=U\cup D(f)$.
    Let $j:V\to \Spec(A)$ be the open immersion, and let $\cA'$ be the integral closure of $\cO_{\Spec(A)}$ in $j_*\cA$.
    Then $j^*\cA'=\cA$, and we can find a coherent subalgebra $\cA''$ of $\cA'$ such that  $j^*\cA''=\cA$,
    which corresponds to a $U$-modification $R''$ of $R$ that satisfies $W^{-1}R''=T$.
    Alternatively, one can use Zariski's Main Theorem, \citestacks{05K0}.
\end{proof}
\begin{Lem}\label{lem:punibrLocalCohomWholeExtend}
    Let $(R,\fm)$ be a Noetherian local domain of 
    dimension at least $2$.
    Assume that $\fm\in \Spec(R)$ is p-unibranch.
    
    Let $M$ be a submodule of $H^1_\fm(R)$.
    Then $R[\fm;M]$ is integral over $R$.
\end{Lem}
\begin{proof}
    Let $M_1$ be the socle of $H^1_\fm(R)$, which is finite as $H^1_\fm(R)$ is Artinian \cite[Chapter V, Corollary 6.5]{Residues-Duality}. 
    The ring $R_1:=R[\fm;M_1]$ is then finite over $R$ by Lemma \ref{lem:onestepExtend}.
    As $\fm\in \Spec(R)$ is p-unibranch, $R_1$ is a local ring, and we denote its maximal ideal by $\fm_1$.
    We know $\dim R_1=\dim R\geq 2$, 
    therefore the same argument applies to the ring $R_1$,
    and from Discussion \ref{discu:R+MTwoSteps} we see $R[\fm;M_2]$ is finite over $R$ for $M_2=(M_1^a:\fm)$.
    Inductively we see $R[\fm;H^1_\fm(R)]$ is integral over $R$,
    and so is its subring $R[\fm;M]$.
\end{proof}

\begin{Rem}
    Instead of p-unibranchness, the conclusion of Lemma \ref{lem:punibrLocalCohomWholeExtend} holds with the weaker assumption that all maximal ideals of all finite $D(\fm)$-modifications of $R$ have height at least $2$.
    This is true, for instance, when $R$ is universally catenary, by the dimension formula \citestacks{02IJ}.
\end{Rem}

\section{Eliminating the obstructions}
\begin{Lem}\label{lem:FONSIisoverFONSI}
    Let $R\subseteq R'$ be a finite inclusion of Noetherian integral domains.
    Then for every $P'\in\OS{R'}$ we have $P'\cap R^\wedge\in\OS{R}$.
    In particular, if $\OS{R}=\emptyset$, then $\OS{R'}=\emptyset$.
\end{Lem}
\begin{proof}
    The proof is similar to the proof of Lemma \ref{lem:NoFonsiht1Preim}.
    
    Let $P'\in\OS{R'}$. 
    Take $P_0'\in\Min(R'^\wedge)$ contained in $P'$ with $\HT(P'/P_0')=1$.
    As $R^\wedge$ is universally catenary,
    we have $\HT(P'\cap R^\wedge/P_0'\cap R^\wedge)=1$ by the dimension formula \citestacks{02IJ}.
    By Lemma \ref{lem:finiteInclMinoverMinAssoverAss}
    $P_0'\cap R^\wedge\in\Min(R^\wedge)$.
    We have $\HT(P'\cap R)\geq \HT(P'\cap R')>1$,
    by the fact $R\to R'$ is integral and by the definition of $\OS{R'}$.
    Therefore $P'\cap R^\wedge\in \OS{R}.$
\end{proof}

The following result is \cite[Theorem 31.1]{Mat-CRT}, and also follows from \cite[Proposition 6.10.6]{EGA4_2}.
\begin{Thm}\label{thm:finiteFailureofCatnery}
    Let $R$ be a Noetherian ring, $\fp\in\Spec(R)$.
    Then there exist at most finitely many $\fP\in V(\fp)$ such that $\HT(\fP/\fp)=1$ and that $\HT(\fP)>\HT(\fp)+1$.
\end{Thm}

\begin{Lem}\label{lem:FONSIfinite}
    Let $R$ be a Noetherian semilocal domain.
    Then $\OS{R}$ is finite.
\end{Lem}
\begin{proof}
    This follows from Theorem \ref{thm:finiteFailureofCatnery}
    as $R^\wedge$ has only finitely many minimal primes.
\end{proof}

The next result is the key technical lemma towards the main theorems.
The proof idea is to ``split up'' a FONSI with a finite extension in $R^\sigma$,
using disconnectedness (Remark \ref{Rem:S2noFONSI}) and the extensions discussed in \S\ref{sec:ExtendbyLocCoh},
and to note this process must terminate.

\begin{Thm}\label{thm:killFONSI}
    Let $R$ be a Noetherian semilocal domain. 
    Then there exists a finite subalgebra $R'$ of $R^\sigma$ such that $\OS{R'}=\emptyset$.
\end{Thm}
\begin{proof}
By Lemmas \ref{lem:FONSIisoverFONSI} and \ref{lem:FONSIfinite} it suffices to find a finite subalgebra $R'$ of $R^\sigma$ for each $P\in\OS{R}$ such that no FONSIs of $R'$ are above $P$.
Fix a $P\in\OS{R}$ and let $\fp=P\cap R\in\Spec(R)$, so we have $\HT(\fp)>1$.

By Theorem \ref{thm:Krull}, there exists a finite $D(\fp)$-modification $R_1$ of $R$ such that all preimages 
of $\fp$ in $\Spec(R_1)$ are p-unibranch.
We know $R_1\subseteq R^\sigma$ as $\HT(\fp)>1$. 
Next, we take a finite $D(\fp R_1)$-modification $R_2$ of $R_1$,
so that the number of maximal ideals of $(R_2^\wedge)_{P}$ (\emph{i.e.} the number of preimages of $P$ in $\Spec(R_2^\wedge)$) is maximal among all possible $R_2$;
 to see this is achievable,
note that $(R_2^\wedge)_{\operatorname{red}}$ is finite birational over $(R_1^\wedge)_{\operatorname{red}}$, so $((R_2^\wedge)_{\operatorname{red}})_{P}$ is contained in the normalization of $((R_1^\wedge)_{\operatorname{red}})_{P}$,
which is finite as a complete Noetherian local ring is Nagata \citestacks{0335},
so the number of maximal ideals of $(R_2^\wedge)_{P}$ is bounded.
Note that we still have $R_2\subseteq R^\sigma$ and that all preimages of $\fp$ in $\Spec(R_2)$ are p-unibranch.
We will show  no FONSIs of $R_2$ are above $P$.

Assume that there exists a $P_2\in \OS{R_2}$ above $P$. 
Let $\fp_2=P_2\cap R_2$, so $\fp_2$ lies above $\fp$, therefore is p-unibranch; and $\HT(\fp_2)>1$ as $P_2\in \OS{R_2}$.
In particular $T:=(R_2)_{\fp_2}$ and $A:=(R_2^\wedge)_{P_2}$ have depth at least $1$.

The punctured spectrum $U$ of the ring $A$ is disconnected (Remark \ref{Rem:S2noFONSI}),
therefore $\cO(U)$ has a nontrivial idempotent $e$.
In the notations of Discussion \ref{discu:DefR+M},
there exists a finite submodule $N\subseteq H^1_{P_2A}(A)$ such that $A+_{P_2A}N=A[P_2A;N]=A[e]$.
As $T\to A$ is flat and as $H^0_{\fp_2T}(T)=0$, we have $H^1_{P_2A}(A)=H^0_{P_2A}(H^1_{\fp_2T}(T)\otimes_T A)$ (Discussion \ref{discu:R+MdifferentIdeals}).
Therefore there exists a finite submodule $M\subseteq H^1_{\fp_2T}(T)$ such that $N\subseteq M\otimes_T A$.
Note that $\fp_2T\in\Spec(T)$ is p-unibranch (Lemma \ref{lem:extendUmodif}), so
the ring $T[\fp_2T;M]$ is finite over $T$ (Lemma \ref{lem:punibrLocalCohomWholeExtend}), and therefore a finite $D(\fp_2T)$-modification.
We can then find a finite $D(\fp_2)$-modification $R_3$ of $R_2$,
which is automatically a finite $D(\fp R_1)$-modification of $R_1$,
such that $(R_3)_{\fp_2}=T[\fp_2T;M]$ (Lemma \ref{lem:extendUmodif}).
Then $(R_3^\wedge)_{P_2}=T[\fp_2T;M]\otimes_T A=A[\fp_2A;M\otimes_T A]\supseteq A[\fp_2A;N]=A[P_2A;N]=A[e]$,
where we used compatibilities in Discussions \ref{discu:DefR+M} and \ref{discu:R+MdifferentIdeals}.
In particular, $(R_3^\wedge)_{P_2}$ is not local, so $P_2$ has more than $1$ preimages in $\Spec(R_3^\wedge)$, so $P$ has more preimages in $\Spec(R_3^\wedge)$ than  in $\Spec(R_2^\wedge)$,
contradicting maximality.
\end{proof}

The following result could have been established along the way;
we derive it formally from the theorem.
\begin{Cor}\label{cor:noFONSIisNS2integral}
    Let $R$ be a Noetherian integral domain.
    Then $\OS{R}=\emptyset$ if and only if $R^{n\sigma}$ is integral over $R$.
    In particular, if $\OS{R}=\emptyset$,
    then $\OS{W^{-1}R}=\emptyset$ for all multiplicative subsets $W$ of $R$.
\end{Cor}
\begin{proof}
    The ``in particular'' statement follows from Lemma \ref{lem:NS2localization}.

    That  $\OS{R}=\emptyset$ implies $R^{n\sigma}$ being integral over $R$ is Theorem \ref{thm:NS2ofFONSIfree}\ref{NS2:integral}.
    Assume $\OS{R}\neq\emptyset$.
    We show $R^{n\sigma}$ is not integral over $R$.
    By Lemma \ref{lem:NS2localization} we may assume $R$ is local.
    Let $R'$ be as in Theorem \ref{thm:killFONSI}.
    Let $P\in \OS{R}$ and let $P_0\in\Min(R^\wedge)$ be contained in $P$ so that $\HT(P/P_0)=1$.
    Write $\fp=P\cap R$, so $\HT(\fp)>1$.

Note that $R^\wedge\to R'^\wedge$ is finite injective.
    Let $P_0'\in\Min(R'^\wedge)$
    be above $P_0$.
    Let $P'\in V(P_0')$ be above $P$.
    As $R^\wedge$ is universally catenary,
    we have $\HT(P'/P_0')=\HT(P/P_0)=1$.
    As $\OS{R'}=\emptyset$,
    we have $\HT(\fp')\leq 1$,
    where $\fp'=P'\cap R'$.
    As $\fp'\cap R=\fp\neq0$ we have $\HT(\fp')=1$.

    Let $\Sigma$ be the finite set of all preimages of $\fp$ in $\Spec(R')$
    and let $\Sigma_0=\Sigma\setminus\{\fp'\}$.
    Let $W'=R'\setminus \cup\Sigma_0$.
    We claim that $W'^{-1}R'\subseteq (R_\fp)^{n\sigma}$.
    This tells us $(R_\fp)^{n\sigma}$ is not integral over $R_\fp$ (as $\Spec(W'^{-1}R')\to \Spec(R'_\fp)$
    is not surjective),
    and Lemma \ref{lem:NS2localization} 
    gives $R^{n\sigma}$ is not integral over $R$.

    To see the claim, let $\fq\in\Spec_1(R)$ be contained in $\fp$.
    We need to show $W'^{-1}R'\subseteq R_\fq$.
    As $\HT(\fp)>1$, we can find $x\in\fp\setminus\fq$.
    Let $\beta\in W'^{-1}R'$.
    Since $\HT(\fp')=1$,
    we know the fraction field of $R'_{\fp'}$, which is also the fraction field of $R$ and $R'$, is equal to $R'_{\fp'}[\frac{1}{x}]$.
    Therefore $x^N\beta\in R'_{\fp'}$ for some $N$,
    so $x^N\beta\in R'_{\fp}$ as $\beta\in W'^{-1}R'$. 
    As $R'\subseteq R^\sigma$,
    we have $R_\fq=R'_{\fq}$.
    Thus $x^N\beta\in R_\fq$,
    so $\beta\in R_\fq$ as $x\not\in\fq$.
\end{proof}

The following are the main results on $(S_2)$-closures.
\begin{Thm}\label{thm:RsIsS2}
    Let $R$ be a Noetherian integral domain.
    Then $R^\sigma$ is $(S_2)$.
\end{Thm}
\begin{proof}
By Lemmas \ref{lem:NS2localization} and \ref{lem:S2localization},
    we may assume $R$ is local.
    By Theorem \ref{thm:killFONSI} we can find a finite $R'\subseteq R^\sigma$ so that $\OS{R'}=\emptyset$.
    Let $\cR$ be the set of all finite $R'$-subalgebras of $R^\sigma$.
    For every $R''\in\cR$ we have $R''^{\sigma}\subseteq R^{\sigma}$ by Lemma \ref{lem:reverseIncl},
    so $R^{\sigma}=\bigcup_{R''\in\cR}R''^{\sigma}$,
    and this union is filtered by Lemmas \ref{lem:FONSIisoverFONSI} and \ref{lem:NoFonsiht1Preim}.
    By Lemma \ref{lem:FONSIisoverFONSI} and Theorem \ref{thm:NS2ofFONSIfree}, each $R''^{\sigma}$ is $(S_2)$.
    We conclude by Lemma \ref{lem:S2filteredunion}. 
\end{proof}

\begin{Cor}
   Let $R$ be a Nagata integral domain. Then there exists a (finite) subalgebra $R'$ of $R^\nu$ that is $(S_2)$ and satisfies $R_\fp=R'_{\fp}$ for all $\fp\in\Spec_1(R)$.
\end{Cor}
Similarly we have
\begin{Thm}\label{thm:SsisS2}
    Let $R$ be a Noetherian integral domain, $S$ a subalgebra of $R^\nu$.
    Then $S^\sigma$ is $(S_2)$.
\end{Thm}
\begin{proof}
    By Lemmas \ref{lem:NS2localization} and \ref{lem:S2localization},
    we may assume $R$ is local.
    Let $\Sigma=\{\fq\in\Spec_1(S)\mid \HT(\fq\cap R)>1\}$.
    Then by Lemmas \ref{lem:NoFonsiht1Preim} and \ref{lem:FONSIfinite} and Theorem \ref{thm:Krull},
    $\Sigma$ is finite.
    We may therefore find a finite subalgebra $R'$ of $S$ such that $\HT(\fq\cap R')=1$ for all $\fq\in\Sigma$,
    Lemma \ref{lem:MinToAss}.
    For all $\fq\in\Spec_1(S)\setminus\Sigma$,
    $\HT(\fq\cap R)=1$,
    so $\HT(\fq\cap R')=1$ as $R\to R'$ is integral.
    Therefore $\HT(\fq\cap R')=1$ for all $\fq\in\Spec_1(S)$,
    consequently $\HT(\fq\cap R'')=1$ for all $\fq\in\Spec_1(S)$ and all $R''\in \cR:=$ the set of all finite $R'$-subalgebras of $S$.
    As in the proof of Corollary \ref{cor:NS2ofFONSIfreeS},
    we see $S^{n\sigma}=\bigcup_{R''\in\cR}R''^{n\sigma}$,
    so $S^{\sigma}=\bigcup_{R''\in\cR}R''^{\sigma}$
    as $S^\nu=R''^\nu=R^\nu$ for all $R''\in\cR$.
    Again, as in the proof of Corollary \ref{cor:NS2ofFONSIfreeS},
    this union is filtered,
    and we conclude by Lemma \ref{lem:S2filteredunion} and Theorem \ref{thm:RsIsS2}.
\end{proof}

\section{Semi-Nagata rings}\label{sec:SagataRings}
\begin{Def}\label{def:semi-Nagata}
    A ring $R$ is \emph{semi-Nagata} if $R$ is Noetherian and, for all finite ring maps $R\to B$ where $B$ is an integral domain, there exists a finite inclusion $B\subseteq C$ of integral domains such that $C$ is $(S_2)$.
\end{Def}
We will show that we can take $C$ inside $B^\sigma$,
Theorem \ref{thm:semi-Nagatacharacterize}.
Therefore Definition \ref{def:semi-Nagata} is equivalent to the definition given in the introduction.
\begin{Rem}\label{Rem:1dimsemi-Nagata}
    A one-dimensional Noetherian ring is semi-Nagata, as we can take $C=B$.
\end{Rem}
\begin{Rem}\label{Rem:Nagatasemi-Nagata}
    A Nagata ring is semi-Nagata, as we can take $C$ to be the normalization of $B$.
\end{Rem}
\begin{Rem}[cf. {\cite{Greco-excellent-finite}}]\label{Rem:finiteInclsemi-Nagata}
    Let $R\to R'$ be a finite map of Noetherian rings.
    If $R$ is semi-Nagata, so is $R'$;
    if $\Spec(R')\to\Spec(R)$ is surjective and $R'$ is semi-Nagata,
    so is $R$.
    This is trivial from our definition.
\end{Rem}
\begin{Rem}\label{Rem:localizesemi-Nagata}
    For a Noetherian ring $R$, a multiplicative subset $W$ of $R$, and a finite ring map $W^{-1}R\to C$ where $C$ is an integral domain,
    there exists a finite ring map $R\to B$  where $B$ is an integral domain and $W^{-1}B=C$.
    Indeed, let $B_0$ be the integral closure of $R$ in $C$.
    Then $W^{-1}B_0=C$, so $W^{-1}B=C$ for some finite subalgebra $B$.
    This tells us a localization of a semi-Nagata ring is semi-Nagata.
\end{Rem}

Following Grothendieck \citestacks{0BIR}, we say a Noetherian ring $R$ is an \emph{$(S_1)$-ring} if all formal fibers of $R$ are $(S_1)$.
We use the fact that the property $(S_1)$ satisfies the axiomatic properties \citestacks{0BIY}.
An essentially finitely generated algebra over an $(S_1)$-ring is an $(S_1)$-ring, \citestacks{0BIV}. 

\begin{Lem}\label{lem:semi-NagataIsS1ring}
    A semi-Nagata ring is an $(S_1)$-ring.
\end{Lem}
\begin{proof}
    Let $R$ be a semi-Nagata local ring.
    We need to show the fibers of $R\to R^\wedge$ are $(S_1)$.
    This is enough by Remark \ref{Rem:localizesemi-Nagata}.
    
    By Noetherian induction we may assume this is true for all proper quotients of $R$.
    If $R$ were not an integral domain we are done,
    so we may assume $R$ is an integral domain.
    Then there exist a finite inclusion $R\subseteq R'$ of integral domains so that $R'$ is $(S_2)$.

    Let  $x\in R^\circ$ be a noninvertible element. 
    Then $R'/xR'$ is $(S_1)$. 
    As the fibers of $R/xR\to (R/xR)^\wedge$ are $(S_1)$ by the induction hypothesis, 
    so are  the fibers of $R'/xR'\to (R'/xR')^\wedge$.
    Therefore $(R'/xR')^\wedge$ is $(S_1)$ \citestacks{0339},
    so $R'^\wedge$ is $(S_1)$ \cite[Proposition 3.4.4]{EGA4_2}.
    As $R\subseteq R'$ is a finite inclusion of Noetherian semilocal domains
    we see $R^\wedge$ is $(S_1)$ (Lemma \ref{lem:finiteInclMinoverMinAssoverAss}).
\end{proof}

\begin{Lem}\label{lem:completeS1finiteS2ify}
    Let $R$ be a Noetherian semilocal domain such that $R^\wedge$ is $(S_1)$.
    Then there exists a finite subalgebra of $R^\sigma$ that is $(S_2)$.
\end{Lem}
\begin{proof}
     By Theorem \ref{thm:killFONSI} we can find a finite subalgebra $R'$ of $R^\sigma$ 
    so that $\OS{R'}=\emptyset$.
    $R'^\wedge$ is $(S_1)$ by Lemma \ref{lem:finiteInclMinoverMinAssoverAss}.
    Therefore $R'^\sigma$ is a finite $R'$-algebra inside $R^\sigma$ (Lemma \ref{lem:reverseIncl}) that is $(S_2)$ (Theorem \ref{thm:NS2ofFONSIfree}).
\end{proof}
    
\begin{Thm}\label{thm:semi-Nagatacharacterizelocal}
    Let $R$ be a Noetherian semilocal ring.
    Then the following are equivalent.
    \begin{enumerate}[label=$(\roman*)$]
        \item\label{semi-Nagatalocal:semi-Nagata} $R$ is semi-Nagata.
  \item\label{semi-Nagatalocal:finiteS2ify} For every finite ring map $R\to B$ where $B$ is an integral domain,
  there exists a finite subalgebra $C$ of $B^\sigma$ that is $(S_2)$.

  \item\label{semi-Nagatalocal:S1ring}
  $R$ is an $(S_1)$-ring.
\end{enumerate}
\end{Thm}
\begin{proof}
    Lemma \ref{lem:semi-NagataIsS1ring} gives \ref{semi-Nagatalocal:semi-Nagata} implies \ref{semi-Nagatalocal:S1ring},
    whereas \ref{semi-Nagatalocal:finiteS2ify} trivially implies \ref{semi-Nagatalocal:semi-Nagata}.
Finally, to see \ref{semi-Nagatalocal:S1ring} implies \ref{semi-Nagatalocal:finiteS2ify},
    we may replace $R$ by $B$ 
    and assume $R=B$ is an integral domain.
    Then $R^\wedge$ is $(S_1)$ by \citestacks{0339},
    so \ref{semi-Nagatalocal:finiteS2ify} follows from Lemma \ref{lem:completeS1finiteS2ify}.
\end{proof}

\begin{Exam}\label{exam:semi-NagataRsigmaNotFinite}
    It is possible that $R^\sigma$ is not finite over $R$,
    even when $R$ is semi-Nagata.
    As in Example \ref{exam:LechFONSI},
    there is a Noetherian local domain $(R,\fm)$ of dimension $2$ such that $R^\wedge\cong k[[x,y,z]]/(x^2,y^2)\cap (z)$,
    where $k$ is a field.
    Then $k[[x,y,z]]/(x^2,y^2)\times k[[x,y]]$
    is a finite $D(\fm R^\wedge)$-modification of $R^\wedge$,
    thus isomorphic to $R'^\wedge$ where $R'$ is a finite $D(\fm)$-modification of $R$ \citetwostacks{0ALK}{05EU}.
By \cite[Proposition 6.3.8]{EGA4_2} the Cohen--Macaulay ring $R'$ is an $(S_1)$-ring, so $R'$ is semi-Nagata, thus so is $R$.
On the other hand,
for the maximal ideal $\fm'$ of $R'$ of height $1$,
the normalization $T$ of $R'_{\fm'}$ is contained in a localization of $R^\sigma$,
as $\HT(\fm'\cap R)=2$.
Since $T$ is not finite over $R'_{\fm'}$ (otherwise the nonreduced ring $R'^\wedge_{\fm'}\cong k[[x,y,z]]/(x^2,y^2)$ will be a subring of a finite product of DVRs),
we see $R^\sigma$ is not finite over $R$.

We remark that the ring $R$ in Example \ref{exam:LechFONSI} is also semi-Nagata for the same reason.
\end{Exam}

\begin{Rem}\label{rem:NoModuleS2ify}
    In general, it cannot be expected that $C$ as in Theorem \ref{thm:semi-Nagatacharacterizelocal}\ref{semi-Nagatalocal:finiteS2ify} is $(S_2)$ as a $B$-module.
    In fact, for the semi-Nagata ring $R$ in Example \ref{exam:LechFONSI} (or \ref{exam:semi-NagataRsigmaNotFinite})
   there exists no  inclusion of finite $R$-modules  $R\to M$ so that $M$ is $(S_2)$.
   To see this, as $\dim R=2$ we have $\depth M\geq 2$
   (so $\depth(M^\wedge)\geq 2$).
   At this point, we can apply \citestacks{00NM} to see $R$ universally catenary, so $R^\wedge$ is equidimensional \citestacks{0AW6}, 
   which is a contradiction.
   For a more explicit examination of what failed, 
   by \citetwostacks{0AVZ}{0DWR} we see $M^\wedge=\Gamma(U^\wedge,\cF^\wedge)$,
   where $U^\wedge$ is the punctured spectrum of $R^\wedge$ and  $\cF^\wedge$ is the sheaf on $\Spec(R^\wedge)$ associated with $M^\wedge$.
   Then for the minimal prime $P_0$ of $R^\wedge$ with $\dim(R^\wedge/P_0)=1$,
   we see $(M^\wedge)_{P_0}$
   is a direct factor of $M^\wedge$.
   Therefore $(M^\wedge)_{P_0}$, and its submodule $(R^\wedge)_{P_0}$, is finite over $R^\wedge$; in other words, $k((z))$ is finite over $k[[z]]$, contradiction.

   Note that, as noted in Example \ref{exam:LechFONSI},
   when the chacteristic of $k$ is zero, we can even make the ring $R$ quasi-excellent.
\end{Rem}

\begin{Rem}\label{rem:UCS2ring=S2module}
    On the other hand, if $R$ is a Noetherian universally catenary domain and $R'$ is an integral domain containing and finite over $R$,
    then $R'$ is $(S_2)$ as a ring if and only if $R'$ is $(S_2)$ as an $R$-module.
    To see this, let us assume $(R,\fm)$ local,
    so $R'$ is semilocal with maximal ideals $\fm'_1,\ldots,\fm'_n$.
    Since $R$ is universally catenary we know $\HT(\fm'_i)=\dim R$ for every $i$,
    see \citestacks{02IJ}.
    Moreover, as $\fm'_i$ are the only preimages of $\fm$ in $\Spec(R')$,
    we have $\min_i\depth R'_{\fm'_i}=\depth_R R'$.
    We conclude that $\depth_R R'\geq\max\{2,\dim R\}$
    if and only if $\depth R'_{\fm'_i}\geq\max\{2,\HT(\fm'_i)\}$ for all $i$.
    This gives the original assertion via localization.
\end{Rem}

\begin{Exam}
    Similar to Example \ref{exam:LechFONSI},
    there exists a Noetherian local domain $(R,\fm)$ of dimension $2$ so that $R^\wedge\cong k[[x,y,z]]/(x^2,y^2)\cap (x)$
    is not $(S_1)$.
    Then $R$ is not semi-Nagata.
\end{Exam}

We say a Noetherian ring $R$ is \emph{$(S_2)$-2} if for every $\fp\in\Spec(R)$,
there exists an $f\in R\setminus\fp$ so that $(R/\fp)_f$ is $(S_2)$.

\begin{Lem}\label{lem:S2-2rings}
    Let $R$ be an $(S_2)$-2 Noetherian ring.
    Then the following hold.
    \begin{enumerate}[label=$(\roman*)$]
    \item\label{S2-2:open} The $(S_2)$ locus of every finite $R$-module is open.
        \item\label{S2-2:fg} Every  essentially finitely generated $R$-algebra is $(S_2)$-2.
    \end{enumerate}
\end{Lem}
\begin{proof}
    \cite[Proposition 6.11.6]{EGA4_2} gives \ref{S2-2:open}.
    For \ref{S2-2:fg}, it suffices to show for a finite type inclusion of Noetherian domains $R\subseteq B$, if $R$ is $(S_2)$,
    then $B_g$ is $(S_2)$ for some $g\in B^\circ$.
    To see this, we may assume $R\to B$ is flat \citestacks{051R},
    and has Cohen--Macaulay fibers \citestacks{045U}.
    Then $B$ is $(S_2)$ by \citestacks{0339}. 
\end{proof}

\begin{Thm}\label{thm:semi-Nagatacharacterize}
    Let $R$ be a Noetherian ring.
    Then the following are equivalent.
    \begin{enumerate}[label=$(\roman*)$]
        \item\label{semi-Nagata:semi-Nagata} $R$ is semi-Nagata.
  \item\label{semi-Nagata:finiteS2ify} For every finite ring map $R\to B$ where $B$ is an integral domain,
  there exists a finite subalgebra $C$ of $B^\sigma$ that is $(S_2)$.

  \item\label{semi-Nagata:S1ring}
  $R$ is an $(S_1)$-ring and is $(S_2)$-2.
\end{enumerate}

If $R$ is semi-Nagata, then every essentially finitely generated $R$-algebra is semi-Nagata.
\end{Thm}
\begin{proof}
As note in Lemma \ref{lem:S2-2rings} and before Lemma \ref{lem:semi-NagataIsS1ring},
\ref{semi-Nagata:S1ring} is preserved by essentially finitely generated algebras, giving the last assertion.

    To see \ref{semi-Nagata:semi-Nagata} implies  \ref{semi-Nagata:S1ring},
    Lemma \ref{lem:semi-NagataIsS1ring} says a semi-Nagata ring is an $(S_1)$-ring.
    Therefore it suffices to show for a semi-Nagata domain $R$, there exists an $f\in R^\circ$ so that $R_f$ is $(S_2)$.
    Let $R\subseteq R'$ be a finite inclusion so that $R'$ is $(S_2)$.
    Then any $f\in R^\circ$ so that $R'_f$ is flat over $R_f$ works.

    As  \ref{semi-Nagata:finiteS2ify} implies  \ref{semi-Nagata:semi-Nagata},
    it suffices to show \ref{semi-Nagata:S1ring} implies \ref{semi-Nagata:finiteS2ify}.
    Assume \ref{semi-Nagata:S1ring} and assume $R$ is an integral domain.
    We must show that there exists a finite subalgebra $R'$ of $R^\sigma$ that is $(S_2)$.
Let $U$ be the $(S_2)$ locus of $R$, which is open by Lemma \ref{lem:S2-2rings},
and let $\fp\not\in U$, so $\HT(\fp)>1$.
    
    The local ring $R_\fp$ is semi-Nagata by Theorem \ref{thm:semi-Nagatacharacterizelocal},
    so by Lemma \ref{lem:NS2localization} there exists a finite subalgebra $R_1$ of $R^\sigma$ such that $(R_1)_\fp$ is $(S_2)$.
    As $R$ is $(S_2)$-2, the $(S_2)$ locus $U_1$ of the ring $R_1$ is open by Lemma \ref{lem:S2-2rings},
    and we have $\bigcap_{h\in R\setminus\fp}D_{R_1}(f)\subseteq U_1$.
    As the constructible topology of $\Spec(R_1)$ is compact \citestacks{0901},
    we see there exists $f_1\in R\setminus\fp$ such that $(R_1)_{f_1}$ is $(S_2)$.
    Since $R_1\subseteq R^\sigma$ we know $R=R_1$ over $U$,
    so the image $Z_1$ of the non-$(S_2)$ locus of $R_1$ in $\Spec(R)$ is disjoint from $U\cup D(f_1)$.
    If $Z_1\neq\emptyset$, take $\fq\in Z_1$, then similarly the semilocal ring $(R_1)_\fq$ is semi-Nagata,
    and we can find $R_2\subseteq R_1^\sigma\subseteq R^\sigma$ (Lemma \ref{lem:reverseIncl}) so that  the image $Z_2$ of the non-$(S_2)$ locus of $R_2$ in $\Spec(R)$ is disjoint from $U\cup D(f_1)\cup D(f_2)$ and $f_2\in R\setminus\fq$.
As the topological space $\Spec(R)$ is Noetherian, we get our desired $R'$ after finitely many steps.
\end{proof}

We include the following argument for a different perspective.

\begin{proof}[Alternative proof of \ref{semi-Nagata:S1ring} implies \ref{semi-Nagata:finiteS2ify}]
Let $R$ be a Noetherian integral domain that satisfies \ref{semi-Nagata:S1ring}.
We want to show there exists a finite subalgebra of $R^\sigma$ that is $(S_2)$.

Let $\fm\in\Max(R)$.
Then $R_\fm$ is semi-Nagata by Theorem \ref{thm:semi-Nagatacharacterizelocal},
    so by Lemma \ref{lem:NS2localization} there exists a finite subalgebra $R(\fm)$ of $R^\sigma$ such that $(R(\fm))_\fm$ is $(S_2)$.

As $R$ is $(S_2)$-2, the $(S_2)$ locus of the ring $R(\fm)$ is open (Lemma \ref{lem:S2-2rings}).
    As the constructible topology of $\Spec(R(\fm))$ is compact \citestacks{0901}, there exists $f(\fm)\in R\setminus\fm$ such that $R(\fm)_{f(\fm)}$ is $(S_2)$.
    
    Take finitely many $\fm_1,\ldots,\fm_n$ so that $D(f(\fm_i))\ (1\leq i\leq n)$ cover $\Spec(R)$ and let $R_1$ be the $R$-algebra generated by all $R(\fm_i)$.
    Then $R(\fm_i)_{f(\fm_i)}\subseteq (R_1)_{f(\fm_i)}$, so by Remark \ref{Rem:S2noFONSI} and Lemma \ref{lem:FONSIisoverFONSI} we have $\OS{R_1}=\emptyset$.
    As $R_1^\sigma\subseteq R^\sigma$ (Lemma \ref{lem:reverseIncl}) we may replace $R$ by $R_1$ to assume $\OS{R}=\emptyset$.
    Apply the same construction again, we see from Theorem \ref{thm:NS2ofFONSIfree}\ref{NS2:S2unique} (and Lemma \ref{lem:NS2localization}) that $R(\fm_i)_{f(\fm_i)}=R^\sigma_{f(\fm_i)}$ for all $i$,
    so $R^\sigma$ is finite over $R$.    
\end{proof}

\begin{Cor}\label{cor:CMissemi-Nagata}
    A Cohen--Macaulay ring is semi-Nagata. 
\end{Cor}
\begin{proof}
\cite[Proposition 6.3.8]{EGA4_2} tells us a Cohen--Macaulay ring is an $(S_1)$-ring,
and \cite[Proposition 6.11.8 and Remarques 6.11.9]{EGA4_2}   tell us a Cohen--Macaulay ring is $(S_2)$-2.
\end{proof}

\begin{Cor}\label{cor:NOFONSISagataNS2finite}
    Let $R$ be a semi-Nagata integral domain so that $\OS{R}=\emptyset$.
    Then $R^\sigma$ is finite over $R$.
\end{Cor}
\begin{proof}
    Immediate from Theorem \ref{thm:NS2ofFONSIfree}\ref{NS2:S2unique}.
\end{proof}

We remark on $(S_2)$-ification of modules. 
\begin{Thm}\label{thm:naiveS2ofmodulesfiniteforSagataNOFONSI}
    Let $R$ be a semi-Nagata ring.
    Assume $\OS{R/\fp_0}=\emptyset$ for all $\fp_0\in\Min(R)$.

    Let $\Sigma=\{\fp\in\Spec(R)\mid \HT(\fp/\fp_0)=1\text{ for some }\fp_0\in\Min(R)\}$.
    Assume that $\Sigma=\Spec_1(R)$.
    
    Let $M$ be a finite module so that $\Ass(M)=\Min(R)$.    
   Then $N:=\bigcap_{\fp\in\Sigma}M_\fp$ is finite and $(S_2)$.
\end{Thm}
Note that the assumptions on $\Sigma$ and $M$ are satisfied, for example, when $R$ is an integral domain and $M$ is torsion-free.
\begin{proof}
    We will use Theorem \ref{thm:semi-Nagatacharacterize}\ref{semi-Nagata:S1ring}.

    Let $U$ be the locus where $M$ is $(S_2)$.
   Then $U$ is open (Lemma \ref{lem:S2-2rings}),
   and contains $\Sigma$ as  $\Sigma=\Spec_1(R)$ and $\Ass(M)=\Min(R)$.
   Let $j:U\to \Spec(R)$ be the canonical open immersion.
  It is now clear that $N=j_*j^*M$.

Let $\fp_0$ be a minimal prime of $R$,
and let $\overline{R}=R/\fp_0$.
By \citestacks{0BK3} to show $N$ is finite it suffices to show 
for $\fp\in \Spec(R)\setminus U$ containing $\fp_0$,
every $P_0\in\Ass(\overline{R}^\wedge_{\fp})$
satisfies $\dim (\overline{R}^\wedge_\fp/P_0)>1$.
   
   We know $\overline{R}^\wedge_\fp$ is $(S_1)$ as $\overline{R}$ is $(S_1)$ with $(S_1)$ formal fibers.
Moreover,
for every $P_0\in \Min(\overline{R}^\wedge_\fp)$,
we have $\dim (\overline{R}^\wedge_\fp/P_0)>1$
as $\HT(\fp)>1$ and as $\OS{\overline{R}_\fp}=\emptyset$ (Corollary \ref{cor:noFONSIisNS2integral}).
Thus for all $P_0\in \Ass(\overline{R}^\wedge_\fp)$,
we have $\dim (\overline{R}^\wedge_\fp/P_0)>1$.
Therefore $N$ is finite.
It is $(S_2)$ by for example \cite[Th\'eor\`eme 5.10.5]{EGA4_2}.
\end{proof}

\section{Lifting the semi-Nagata property}\label{sec:liftSagataUC}
In this section, we prove the following result,
giving a partial answer to Question \ref{ques:formallift} for the semi-Nagata property.
\begin{Thm}\label{thm:liftSagataUC}
    Let $R$ be a Noetherian ring, $I$ an ideal of $R$.
    Assume that 
    \begin{enumerate}
        \item $R$ is $I$-adically complete.
        \item\label{liftSagata:quotSagata} $R/I$ is semi-Nagata.
        \item\label{liftSagata:UC}
        $R$ is universally catenary.
    \end{enumerate}
    Then $R$ is semi-Nagata.
\end{Thm}

We proceed with the proof. 
We use the characterization Theorem \ref{thm:semi-Nagatacharacterize} without further mentioning.
By Noetherian induction, we may assume
\begin{enumerate}[start=4]
    \item\label{liftSagata:allquotSagata}
    $R/\fa$ is semi-Nagata for all nonzero ideals $\fa$ of $R$.
\end{enumerate}
By Definition \ref{def:semi-Nagata} we may also assume 
\begin{enumerate}[resume]
    \item $R$ is an integral domain, 
\end{enumerate} 
and we only need to find a finite subalgebra $R'$ of $K$, the fraction field of $R$, that is $(S_2).$
As $R$ is $I'$-adically complete for all ideals $I'\subseteq I$, 
by \eqref{liftSagata:allquotSagata} we may assume 
\begin{enumerate}[resume]
    \item $I$ is generated by a single element $f\neq 0$.
\end{enumerate}
Let $\cM$ be the set of minimal prime divisors of $I$.

We will construct a sequence of submodules $R=M_0\subseteq M_1\subseteq\ldots$ of $K$
so that the union $M=\bigcup_i M_i$ is finite over $R$ and that $M/fM$ is $(S_1)$.
We will later show that the existence of such an $M$ implies $R$ is semi-Nagata.

For a finite submodule $X$ of $K$ we let 
$J(X)$ be the intersection of the embedded primes of the $R$-module $X/fX$.
The module $H^0_{J(X)}(X/fX)$ is canonically identified with $H(X):=H^1_{J(X)}(X)[f]$ as $f$ is a nonzerodivisor on $X$,
and we have a module $X^+:=X+_{J(X)}H(X)\subseteq K$ fitting into an exact sequence
\[
    \begin{tikzcd}
    0\arrow[r] & X \arrow[r] & X^+ \arrow[r] & H(X) \arrow[r] & 0
    \end{tikzcd}
\]
similar to the construction $R+_I M$ in Discussion \ref{discu:DefR+M}.
We will show that $M_0=R$ and $M_{i+1}=M_i^+$ gives the desired sequence of modules.

The exact sequence above gives an exact sequence
\[
    \begin{tikzcd}
    0\arrow[r] & H(X) \arrow[r] & X/fX \arrow[r] & X^+/fX^+ \arrow[r] & H(X) \arrow[r] & 0.
    \end{tikzcd}
\]
As $H(X)$ is identified with $H^0_{J(X)}(X/fX)$ we see from primary decomposition that $\Ass\left(\frac{X/fX}{H(X)}\right)$ is the set of minimal prime divisors of $X/fX$.
We know $\Supp(X)=\Spec(R)$ as $X$ is a submodule of $K$,
so $\Supp(X/fX)=\Spec(R/fR)$ by Nakayama's Lemma,
thus $\Ass\left(\frac{X/fX}{H(X)}\right)=\cM$.
It follows that the associated primes of
the module $M/fM=\colim_i M_i/fM_i=\colim_i \frac{M_i/fM_i}{H(M_i)}$
are all in $\cM$,
so $M/fM$ is $(S_1)$ as soon as $M/fM$ is finite.

The discussion above also shows
$\Ass(X^+/fX^+)\subseteq \cM\cup V(J(X))$,
so $J(X^+)\supseteq J(X)$.
As $R$ is Noetherian, there exists $i_0$ so that $J(M_{i})=J(M_{i_0})$ for all $i\geq i_0$.
Let $\cJ$ be the set of minimal prime divisors of $J:=J(M_{i_0})$.
For every $i\geq i_0$ we have an exact sequence
\[
    \begin{tikzcd}
    0\arrow[r] & \frac{M_{i}/fM_{i}}{H(M_{i})} \arrow[r] & M_{i+1}/fM_{i+1} \arrow[r] & H(M_{i}) \arrow[r] & 0
    \end{tikzcd}
\]
which gives an injection $H(M_{i+1})=H^0_J(M_{i+1}/fM_{i+1})\to  H(M_{i})$
as $H^0_J\left(\frac{M_{i}/fM_{i}}{H(M_{i})}\right)=0$.
Thus there exists $i_1\geq i_0$ so that $H(M_{i+1})_{\fP}=H(M_{i})_{\fP}$ for all $\fP\in\cJ$ and all $i\geq i_1$,
as the lengths of $H(M_{i})_{\fP}$ are finite.
It follows that
$$\left(\frac{M_{i_1}/fM_{i_1}}{H(M_{i_1})}\right)_{\fP}=\left(\frac{M_{i_1+1}/fM_{i_1+1}}{H(M_{i_1+1})}\right)_{\fP}=\ldots=(M/fM)_{\fP}$$ for all $\fP\in \cJ$;
note that the same is true for all $\fP\in D(J)$,
as, in that case, $(M_i)_{\fP}=(M_{i+1})_{\fP}$.

We now apply Theorem \ref{thm:naiveS2ofmodulesfiniteforSagataNOFONSI} to the semi-Nagata ring $A=R/fR$ and the module $Y=\frac{M_{i_1}/fM_{i_1}}{H(M_{i_1})}$.
The condition $\Ass(Y)=\Min(A)$ follows from the construction.
The conditions $\OS{A/\fp_0}=\emptyset$ and $\Sigma=\Spec_1(A)$ follows from the condition $R$ is universally catenary, see Remark \ref{rem:UCnoFONSI}.
If $\fP\in D(J)$ or $\fP\in\cJ$ then we know $(M/fM)_{\fP}=Y_{\fP}$.
However, $D(J)\cup \cJ$ covers $\Spec_1(A)$ as $R$ is catenary.
We see $(M/fM)_{\fP}=Y_{\fP}$ for all $\fP\in\Spec_1(A)$,
so Theorem \ref{thm:naiveS2ofmodulesfiniteforSagataNOFONSI} tells us $M/fM$ is finite,
hence $(S_1)$ as noted before.

As $M$ is a submodule of $K$, we know $M\subseteq M_\fp=R_\fp$ for all $\fp\in\cM$,
in particular $M$ is $f$-adically separated.
Therefore $M$ is a submodule of its $f$-adic completion,
which is finite as $R$ is $f$-adically complete and $M/fM$ is finite.

We have found a finite submodule $M$ of $K$ containing $R$ so that $M/fM$ is $(S_1)$.
Let $\fp\in V(f)$.
The fibers of $R_\fp/fR_\fp\to R_\fp^\wedge/fR_\fp^\wedge$ are $(S_1)$
as $R/fR$ is semi-Nagata,
so $M_\fp^\wedge/fM_\fp^\wedge$ is $(S_1)$ by \cite[Proposition 6.4.1]{EGA4_2},
thus $M_\fp^\wedge$ is $(S_1)$ by \cite[Proposition 3.4.4]{EGA4_2}.
We know $R_g=M_g$ for some $g\in R^\circ$,
so $(R_\fp^\wedge)_g=(M_\fp^\wedge)_g$ and $g$ is a nonzerodivisor in both $R_\fp^\wedge$ and $M_\fp^\wedge$,
showing that $R_\fp^\wedge$ is $(S_1)$.

As $\OS{R_\fp}=\emptyset$ (Remark \ref{rem:UCnoFONSI}),
by Theorem \ref{thm:NS2ofFONSIfree}
there exists a finite subalgebra $R(\fp)$ of $R^\sigma$ so that $R(\fp)_\fp$ is $(S_2)$.
Note that this is the same as $R(\fp)_\fp$ is an $(S_2)$ $R_\fp$-module, as $R$ is universally catenary (Remark \ref{rem:UCS2ring=S2module}).
As $R/\fp$ is semi-Nagata and therefore $(S_2)$-2, the proof of \cite[Proposition 6.11.6]{EGA4_2} tells us there exists $h(\fp)\not\in\fp$ so that $R(\fp)_\fP$ is $(S_2)$ for all $\fP\in V(\fp)\cap D(h)$.
As the constructible topology of $R/fR$ is compact \citestacks{0901}
we see there exist finitely many finite subalgebras $R_1,\ldots,R_n$ of $R^\sigma$ so that for every $\fP\in V(f)$ there exists a $j$ such that $(R_j)_\fP$ is $(S_2)$,
in other words, $(R_j)_\fP=(R^\sigma)_\fP$ (Theorem \ref{thm:NS2ofFONSIfree}).
Let $R'$ be the finite subalgebra generated by all $R_j$, so $R'_\fP=(R^\sigma)_\fP$ for all $\fP\in V(f)$,
thus $R'_\fP$ is $(S_2)$ for all $\fP\in V(f)$.
As $R$ is $f$-adically complete,
this tells us $R'_\fP$ is $(S_2)$ for all $\fP\in \Max(R)$,
so $R'$ is $(S_2)$, as desired.

\section{The local lifting argument}\label{sec:Local-Lifting-Nishimura}
In this section, we present an adapted version of Nishimura's argument for local lifting \cite{Nishimura-semilocal-lifting}.
It is a variant of Rotthaus' argument \cite{Rotthaus-qe-semilocal-lifting}, which is axiomitized in \cite{BI-semilocal-lifting}.
A key component of the argument in all the three aforementioned articles
is that the property of concern must imply reducedness (cf. \cite[Remark after Theorem 2.3]{BI-semilocal-lifting}).
We remove this restriction. 

\begin{Discu}\label{discu:Properties}
    Let $\bP$ be a property of Noetherian rings.
The $\bP$-locus $U_\bP(A)$ of a ring $A$ is the set of $\fp\in\Spec(A)$ so that $A_\fp$ satisfies $\bP$.

A map $\varphi:A\to B$ of Noetherian rings is said to be a  \emph{$\bP$-map} if $\varphi$ is flat with geometrically $\bP$ fibers.
If $\bP$ satisfies \ref{PisPointwise}\ref{Pdescends}\ref{Pascends} below,
then for a $\bP$-map $\varphi$ we always have $\Spec(\varphi)^{-1}(U_\bP(A))=U_\bP(B)$.

A Noetherian ring $A$ is said to be a \emph{$\bP$-ring} if its formal fibers are geometrically $\bP$,
in other words, $A_\fp\to A^\wedge_\fp$ is a $\bP$-map for all $\fp\in\Spec(A)$.
By \citestacks{0BIU}, if $\bP$ satisfies \ref{PisSing}--\ref{Pascends} below,
then a Noetherian ring $A$ is a $\bP$-ring if and only if $A_\fm\to A^\wedge_\fm$ is a $\bP$-map for all $\fm\in\Max(A)$.
When $A$ is semilocal, this is to say $A\to A^\wedge$ is a $\bP$-map (cf. \cite[Proposition 7.3.14]{EGA4_2});
and if $\bP$ satisfies \ref{PisSing}--\ref{Pascends} below,
this is also equivalent to that for every finite $A$-algebra $B$ that is an integral domain,
$(B^\circ)^{-1}B^\wedge$ satisfies $\bP$,
see \cite[Proposition 7.3.16]{EGA4_2}.
By \citestacks{0BIV}, an essentially finitely generated algebra over a $\bP$-ring is a $\bP$-ring.

    Consider the following conditions $\bP$ may satisfy.
    \begin{enumerate}[label=(\Roman*)]
        \item\label{PisSing} Every regular Noetherian ring satisfies $\bP$.
        \item\label{PisPointwise} A Noetherian ring $A$ satisfies $\bP$ if and only if all $A_\fp$, $\fp\in\Spec(A)$, satisfies $\bP$.
        \item\label{Pdescends} For a flat local map $A\to B$ of Noetherian local rings,
        if $B$ satisfies $\bP$, so does $A$.
        \item\label{Pascends} For a local $\bP$-map $A\to B$ of Noetherian local rings,
        if $A$ satisfies $\bP$, so does $B$.
        \item\label{PisOpen} For a Noetherian complete local ring $A$, $U_\bP(A)$ is open.
        \item\label{PGroLocalizes} Let $\varphi:A\to B$ be a flat local map of Noetherian local rings.
        If $A$ is a $\bP$-ring and the closed fiber of $\varphi$ is geometrically $\bP$,
        then $\varphi$ is a $\bP$-map.
    \end{enumerate}
\end{Discu}

\begin{Rem}\label{rem:MurGroLocalize}
Whether or not $\bP$ satisfies \ref{PGroLocalizes} is generally called the Grothendieck localization problem for $\bP$.
The paper \cite{Mur-Grothendieck-localization} provides a uniform treatment of this problem,
and provides a list of references of known results on what properties satisfy \ref{PisSing}--\ref{PGroLocalizes}.
In particular, \ref{PisSing}--\ref{PGroLocalizes} hold for $\bP$=``$(S_k)$,'' ``Cohen--Macaulay,'' ``Gorenstein,'' and ``lci.''
\end{Rem}

\begin{Def}\label{def:CategoryDQP} 
    Let $\bP$, $\bQ$ be two properties of Noetherian rings so that $\bP$ implies $\bQ$ and that both $\bP$ and $\bQ$ satisfy \ref{PisSing}--\ref{PGroLocalizes}.
    Let $\cD^\bQ_{\bP-1}$ and $\cD^\bQ_{\bP}$ be two subcategory of rings described as follows.
    The objects of $\cD^\bQ_{\bP-1}$ are Noetherian rings $A$ that satisfy the following conditions.
    \begin{enumerate}
        \item $A$ satisfies $\bQ$.    
        \item $U_\bP(A)$ is open.
    \end{enumerate}
    The morphisms of $\cD^\bQ_{\bP-1}$ are $\bP$-maps.
    $\cD^\bQ_\bP$ is the full subcategory of $\cD^\bQ_{\bP-1}$ of objects $A$ satisfying
    \begin{enumerate}[resume]
    \item\label{DisPring} $A$ is a $\bP$-ring.
\end{enumerate}
    
    We note that if $\varphi:A\to B$ is a $\bP$-map of Noetherian rings,
    then $A\in\cD^\bQ_{\bP-1}$
    implies $B\in\cD^\bQ_{\bP-1}$;
    if, further, $\varphi$ is faithfully flat,
    then $B\in\cD^\bQ_{\bP-1}$ implies $A\in\cD^\bQ_{\bP-1}$,
    cf. \citestacks{02JY}.

    For a subcategory $\cC$ of $\cD^\bQ_{\bP-1}$,
    a \emph{strictly functorial $\bP$-assignment on $\cC$} is an assignment $A\mapsto \fc(A)$ for all $A\in \cC$ where $\fc(A)$ is a nonzero ideal of $A$ satisfying $V(\fc(A))=\Spec(A)\setminus U_\bP(A)$,
    such that $\varphi(\fc(A))B=\fc(B)$ for all $\varphi:A\to B$ in $\cC$.

When $\bQ$ is the trivial property,
that is, every Noetherian ring satisfies $\bQ$,
we write $\cD_{\bP-1}$ and $\cD_{\bP}$ instead.
\end{Def}

\begin{Rem}\label{rem:BI-Q-is-reduced}
    In \cite{Nishimura-semilocal-lifting,BI-semilocal-lifting},
    $\bQ$=``reduced,''
    and $\fc(A)$ is the unique radical ideal that satisfies $V(\fc(A))=\Spec(A)\setminus U_{\bP}(A)$.
    The lifting of $\bQ$-rings is \cite{Marot-Nagata-Lift}.
    As a reduced ring is $(R_0)$ we see $\fc(A)_\fq=A_\fq$ for all $\fq\in\Min(A)$,
    in particular $\fc(A)\neq 0$. 
    We have $\varphi(\fc(A))B=\fc(B)$ for all $\varphi:A\to B$ in $\cD^\bQ_{\bP-1}$ as
    $\Spec(\varphi)^{-1}(U_\bP(A))=U_\bP(B)$
    and as the fibers of $\varphi$ are reduced,
    so $\varphi(\fc(A))B$ is radical.

    In our case, we do not have such luxury, and it is necessary to find the assignments case-by-case. 
    We will find assignments on $\cD^\bQ_{\bP}$
for $\bQ$ trivial and $\bP$=``$(S_1)$,''
$\bQ$=``$(S_1)$'' and $\bP$=``$(S_2)$,''
and $\bQ$=```$(S_2)$'' and $\bP$=``$(S_k)$'' $(k\geq 3)$, ``Gorenstein,'' and ``lci.''
\end{Rem}

\begin{Thm}\label{thm:Nishimura-local-lifting-argument}
    Let $\bP$, $\bQ$ be two properties of Noetherian rings so that $\bP$ implies $\bQ$ and that both $\bP$ and $\bQ$ satisfy \ref{PisSing}--\ref{PGroLocalizes}.
    Assume that there exists a strictly functorial $(\bP,\bQ)$-assignment on $\cD^\bQ_\bP$.

    Let $R$ be a Noetherian semilocal ring, $I$ an ideal of $R$.
    Assume 
    \begin{enumerate}
        \item $R$ is $I$-adically complete.
        \item\label{Nsmr:quotPring} $R/I$ is a $\bP$-ring.
        \item\label{Nsmr:Qring} $R$ is a $\bQ$-ring.
        \item\label{Nsmr:Q-ify} For every finite $R$-algebra $B$ that is an integral domain,
        there exists a finite inclusion of domains $B\subseteq C$
        such that $C$ satisfies $\bQ$.
    \end{enumerate}
    Then $R$ is a $\bP$-ring.
\end{Thm}
\begin{proof}
    Fix a strictly functorial $\bP$-assignment $A\mapsto\fc(A)$ on $\cD^\bQ_\bP$.
    
    Let $R$ be a ring with an ideal $I$ that satisfy the assumptions.
    For every $R$-algebra $S$ we denote by $S^*$ the $(IS)$-adic completion of $S$.
    By induction, we may assume

    \begin{enumerate}[start=5]
        \item\label{Nsmr:indOnDim} the theorem holds when the dimension of $R$ is strictly smaller.
    \end{enumerate}

    It suffices to show for every $R$-algebra $B$ that is an integral domain,
    $(B^\circ)^{-1}B^\wedge$ satisfies $\bP$.
    By \eqref{Nsmr:Q-ify} 
    we may assume $B$ satisfies $\bQ$. 
    Replace $R$ by $B$ we may  assume
     \begin{enumerate}[resume]
        \item $R$ is an integral domain that satisfies $\bQ$. 
    \end{enumerate}
    By \eqref{Nsmr:Qring}, \ref{PisOpen}, and the fact a complete local ring is a $\bP$-ring we see
    \begin{enumerate}[resume]
        \item\label{Nsmr:RhatisQ} $R^\wedge\in \cD^\bQ_\bP$. 
    \end{enumerate}
    By \eqref{Nsmr:indOnDim}, we have
    \begin{enumerate}[resume]
        \item\label{Nsmr:locPring} $(R_\fp)^*$ is a $\bP$-ring for all $\fp\in\Spec(R)\setminus\Max(R)$.
    \end{enumerate}

Write $C=\fc(R^\wedge)\neq 0$, and for every $n\in\bZ_{\geq 1}$, $C_n=C+I^nR^\wedge$,
$\fa_n=C_n\cap R$.
We will show $C_n=\fa_nR^\wedge$ for all $n$,
which implies $C\cap R\neq 0$ by  \cite[Lemma 2]{Rotthaus-qe-semilocal-lifting},
which then tells us the generic fiber $(R^\circ)^{-1}R^{\wedge}$ is $\bP$,
as $V(C)=\Spec(R^\wedge)\setminus U_\bP(R^\wedge)$.

By consideration of a primary decomposition of $C_n$,
and by the fact flat base change commutes with finite intersections,
it suffices to show for a primary ideal $Q$ containing $C_n$ we have $C_n\subseteq (Q\cap R)R^\wedge$.
If $\sqrt{Q}$ is maximal then this is trivial as $Q=(Q\cap R)R^\wedge$.
Therefore we may assume $\sqrt{Q}$ is not maximal.
Let $\fp=\sqrt{Q}\cap R=\sqrt{Q\cap R}\in\Spec(R)\setminus\Max(R)$.
As $R\to R^\wedge$ is flat and as $Q\cap R$ is $\fp$-primary,
every prime divisor of $(Q\cap R)R^\wedge$ is above $\fp$.
Therefore it suffices to show
$C_n (R^\wedge)_\fp\subseteq (Q\cap R)(R^\wedge)_\fp$.
In the remainder of the proof we show 
$C_n (R^\wedge)_\fp=\fa_n (R^\wedge)_\fp$
for all $\fp\in\Spec(R)\setminus\Max(R)$,
which is enough as $Q\cap R\supseteq C_n\cap R=\fa_n$.

    Consider the commutative diagram of rings
    \[\begin{CD}
        R@>>> R^\wedge\\
        @VVV @VVV\\
        R_\fp @>{f_\fp}>> (R^\wedge)_\fp\\
        @V{g_\fp}VV @V{g^\wedge_\fp}VV\\
        (R_\fp)^* @>{f^*_\fp}>> ((R^\wedge)_\fp)^*.\\
    \end{CD}\]
We know $(R^\wedge)_\fp$ is (quasi-)excellent as $R^\wedge$ is complete,
therefore $((R^\wedge)_\fp)^*$ is quasi-excellent \cite{formal-lifting-excellence-Gabber}.
In particular, both $(R^\wedge)_\fp$  and $((R^\wedge)_\fp)^*$ are $\bP$-rings and have open $\bP$-locus.
By \eqref{Nsmr:locPring} $(R_\fp)^*$ is a local $\bP$-ring,
hence its $\bP$-locus is open by \ref{PisOpen}, cf. \citestacks{02JY}.
    
    We know $f_\fp$ and $g_\fp$ are $\bQ$-maps and $g^\wedge_\fp$ is a $\bP$-map by \eqref{Nsmr:Qring}\eqref{Nsmr:locPring} and \citestacks{0BK9}.
    The map  $f^*_\fp$ is faithfully flat \citestacks{0AGW}, and 
    for every $\fQ\in V(I(R_\fp)^*)$,
    the fiber of $f^*_\fp$ over $\fQ$
    is the same as the formal fiber of $R$ over $\fQ\cap R\in V(I)$,
    which is geometrically $\bP$
    by \eqref{Nsmr:quotPring}.
    As every maximal ideal of $((R^\wedge)_\fp)^*$ contains $I((R^\wedge)_\fp)^*$,
    we see from \ref{PGroLocalizes} that $f^*_\fp$ is a $\bP$-map. 
    Consequently, if we remove $R$ and $R_\fp$, then the diagram above is a diagram in $\cD^\bQ_\bP$ (cf. \eqref{Nsmr:RhatisQ}).
    Therefore $C((R^\wedge)_\fp)^*=\fc(((R^\wedge)_\fp)^*)=\fc((R_\fp)^*)((R^\wedge)_\fp)^*$.

    Now, let 
    $\fb=(\fc((R_\fp)^*)+I^n(R_\fp)^*)\cap R_\fp$,
    so $\fb(R_\fp)^*=\fc((R_\fp)^*)+I^n(R_\fp)^*$
    as $(R_\fp)^*$ is the $I$-adic completion of $R_\fp$.
    Then we have $\fb ((R^\wedge)_\fp)^*=C_n((R^\wedge)_\fp)^*$,
    so $\fb (R^\wedge)_\fp=C_n(R^\wedge)_\fp$
    as both sides contain $I^n (R^\wedge)_\fp$.
    Contract to $R$ we see $\fb=\fa_n R_\fp$, so $\fa_n (R^\wedge)_\fp=C_n(R^\wedge)_\fp$,
    as desired.
    \end{proof}
\begin{Rem}
    We used \cite{formal-lifting-excellence-Gabber} to ensure the ring $((R^\wedge)_\fp)^*$ is a $\bP$-ring.
    A weaker result may be enough.
     
     If lci implies $\bP$, then every lci ring is a $\bP$-ring (cf. \cite[(5.4)]{avramov-ci}).
    The ring $((R^\wedge)_\fp)^*$ is a quotient of a regular ring as $R^\wedge$ is, so it is a $\bP$-ring,
    avoiding \cite{formal-lifting-excellence-Gabber}.
    This is the case in our applications in \S\ref{sec:Local-lifting-result}.
    
    We needed to do this because our $\fc(-)$ is only defined on $\cD^\bQ_{\bP}$.
    In \cite{Nishimura-semilocal-lifting,BI-semilocal-lifting},
    $\fc(-)$ is defined on the whole of $\cD^\bQ_{\bP-1}$ (Remark \ref{rem:BI-Q-is-reduced}),
    so this is unnecessary.
\end{Rem}

\section{Extending $\bP$-assignments}\label{sec:ExtendAssignment}
\begin{Def}\label{def:CategoryCompleteQP}
    Let $\bP,\bQ$, $\cD^\bQ_\bP$ be as in Definition \ref{def:CategoryDQP}.
    For an integer $d\geq 0$ let $^d\cA^\bQ_\bP$
    be the full subcategory of $\cD^\bQ_\bP$ of rings $A\in\cD^\bQ_\bP$ that are complete local of dimension $d$
    whose $\bP$-locus is the punctured spectrum.
    Let $\cA^\bQ_\bP$ be the  disjoint union of all $^d\cA^\bQ_\bP$.
    In other words,
    the objects are $\cA^\bQ_\bP$ are complete local rings in $\cD^\bQ_\bP$ whose $\bP$-locus is the punctured spectrum,
    and the morphisms are local $\bP$-maps whose closed fiber has dimension $0$.

    For every $A\in \cA^\bQ_\bP$,
    denote by $\fm_A$ the maximal ideal of $A$.
    We know $\Spec(A)\setminus U_\bP(A)=\{\fm_A\}$.
    Therefore an ideal $\fc$ satisfying $V(\fc)=\Spec(A)\setminus U_\bP(A)$ is the same as $\fc$ being $\fm_A$-primary.

    When $\bQ$ is trivial we write $^d\cA_\bP$ and $\cA_\bP$ instead.
\end{Def}
\begin{Lem}\label{lem:extendPassignment}
    Let $\bP,\bQ$, $\cD^\bQ_\bP,\cA^\bQ_\bP$ be as in Definitions \ref{def:CategoryDQP} and \ref{def:CategoryCompleteQP}.
    Assume that $\bP$ implies $(S_1)$.
    Then every strictly functorial $\bP$-assignment $\fc(-)$ on $\cA^\bQ_\bP$
    extends uniquely to a strictly functorial $\bP$-assignment $\fc(-)$ on $\cD^\bQ_\bP$
    in a way that $\fc(A)$ has no embedded prime divisors for all $A\in\cD^\bQ_\bP$.
\end{Lem}
\begin{proof}
    Let $\fc(-)$ on $\cA^\bQ_\bP$ be a given strictly functorial $\bP$-assignment.
    
    For $A\in \cD^\bQ_{\bP}$,
    let $\fp_1,\ldots,\fp_n\ (n\geq 0)$ be the generic points of $\Spec(A)\setminus U_\bP(A)$.
    If a desired extension exists,
    then it must satisfy $\fc(A)=\bigcap_i(\fc(A_{\fp_i})\cap A)$ as $\fc(A)$ has no embedded prime divisors.
    Furthermore,
    we must have $\fc(A_{\fp_i})=\fc(A^\wedge_{\fp_i})\cap A_{\fp_i}$,
    as the completion map $A_{\fp_i}\to A^\wedge_{\fp_i}$ is in $\cD^\bQ_\bP$ (\emph{i.e.} a $\bP$-map) by condition \eqref{DisPring} in Definition \ref{def:CategoryDQP}.
    As $\fp_i$ is a generic point of $\Spec(A)\setminus U_\bP(A)$
    we see $U_\bP(A_{\fp_i})=D(\fp_iA_{\fp_i})$,
    therefore $U_\bP(A^\wedge_{\fp_i})=D(\fp_iA^\wedge_{\fp_i})$,
    in other words $A^\wedge_{\fp_i}\in \cA^\bQ_\bP$.
    This shows the uniqueness of the extension; we must have $\fc(A)=\bigcap_i\fc(A^\wedge_{\fp_i})\cap A$.

    It remains to verify that $\fc(A):=\bigcap_i\fc(A^\wedge_{\fp_i})\cap A$ is indeed a strictly functorial $\bP$-assignment;
    by construction it has no embedded prime divisors
    as each $\fc(A^\wedge_{\fp_i})\cap A$ is $\fp_i$-primary.
    It is clear that $V(\fc(A))=\Spec(A)\setminus U_\bP(A)$.
    We have $\fc(A)\neq 0$ as $\fc(A)=A$ when $n=0$, and $\fc(A)_{\fp_1}=\fc(A^\wedge_{\fp_1})\cap A_{\fp_1}\neq 0$ when $n>0$, as $\fc(A^\wedge_{\fp_1})$ is nonzero and $(\fp_1A^\wedge_{\fp_1})$-primary.

    It remains to show for $\varphi:A\to B$ in $\cD^\bQ_\bP$,
    we have $\fc(A)B=\fc(B)$,
    where $\varphi$ is omitted in the notation.
    Let $\fq_{ij}\ (1\leq j\leq m_i)$ be the minimal prime divisors of $\fp_i B$, where $m_i\geq 0$.
    As $\varphi$ is a $\bP$-map, $\Spec(\varphi)^{-1}(U_\bP(A))=U_\bP(B)$,
    so $\fq_{ij}\ (1\leq j\leq m_i,1\leq i\leq n)$
    are exactly the generic points of $\Spec(B)\setminus U_\bP(B)$.
    Moreover, as $\bP$ implies $(S_1)$,
    we see for $\fc_i:=\fc(A^\wedge_{\fp_i})\cap A$,
    $\Ass_{B}(B/\fc_iB)=\{\fq_{ij}\mid 1\leq j\leq m_i\}$.
    Therefore it suffices to show
    $\fc_iB_{\fq_{ij}}=\fc(B^\wedge_{\fq_{ij}})\cap B_{\fq_{ij}}$,
    and as both sides are $\fq_{ij}$-primary, passing to the completion we see it suffices to show $\fc(A^\wedge_{\fp_i})B^\wedge_{\fq_{ij}}=\fc(B^\wedge_{\fq_{ij}})$.
    We know $A^\wedge_{\fp_i},B^\wedge_{\fq_{ij}}\in {^d\cA^\bQ_\bP}$
    for $d:=\HT(\fp_i)=\HT(\fq_{ij})$,
    therefore, as our $\fc(-)$ is strictly functorial on $\cA^\bQ_\bP$,
    it suffices to show $A^\wedge_{\fp_i}\to B^\wedge_{\fq_{ij}}$ is a $\bP$-map.
    By \ref{PGroLocalizes} it suffices to show 
    $\kappa(\fp_i)\to (B/\fp_iB)^\wedge_{\fq_{ij}}$
    is a $\bP$-map.
    This follows from the fact $\varphi:A\to B$ is a $\bP$-map
    and the fact $B/\fp_iB$,
    a quotient of $B\in\cD^\bQ_\bP$,
    is a $\bP$-ring.
\end{proof}

In the next two sections, we will find assignments on $\cA^\bQ_{\bP}$
for $\bQ$ trivial and $\bP$=``$(S_1)$,''
$\bQ$=``$(S_1)$'' and $\bP$=``$(S_2)$,''
and $\bQ$=```$(S_2)$'' and $\bP$=``$(S_k)$'' $(k\geq 3)$, ``Gorenstein,'' and ``lci.''

\section{$(S_k)$-, Cohen--Macaulay-, and Gorenstein-assignments}\label{sec:SkCMGorAssign}

\begin{Lem}\label{lem:S1assignment}
    Let $\bP$=``$(S_1)$.''
    Then $\fc(A)=\Ann_A(H^0_{\fm_A}(A))$
    is a strictly functorial $\bP$-assignment on $\cA_\bP$.
\end{Lem}
\begin{proof}
    It is clear that $\fc(A)$
    is $\fm_A$-primary and $\fc(-)$ is strictly functorial,
    as the closed fiber of all maps in $\cA_{\bP}$ have dimension $0$.
    As an Artinian ring is $(S_1)$ we see $^0\cA^\bQ_\bP=\emptyset$,
    so $\dim A>0$ and any $\fm_A$-primary ideal is nonzero.
\end{proof}

\begin{Discu}\label{discu:cofS2}
   Let $\bQ$=``$(S_1)$'' and $\bP$=``$(S_2)$.''
    For $A\in \cA^\bQ_\bP$,
    let $\Sigma_1$ be the set of primary components $Q$ of $0$
    so that $\dim(A/Q)=1$,
    and let $\Sigma_2$ be the set of primary components $Q$ of $0$
    so that $\dim(A/Q)>1$.
    Note that $A$ is $(S_1)$ and not $(S_2)$, so $0$ has no embedded primes and $\dim A>1$, therefore primary components of $0$ are uniquely determined and $\Sigma_2\neq\emptyset$.
    Let
    $A_1=A/\bigcap_{Q\in\Sigma_1}Q,A_2=A/\bigcap_{Q\in\Sigma_2}Q$.
    Let $U_2$ be the punctured spectrum of $A_2$ and let $A_2'=\cO(U_2)$.
    As $\dim(A/Q)>1$ for all $Q\in\Sigma_2$,
    similar to Lemma \ref{lem:naiveS2finiteforcomplete} (cf. \cite[Lemma 2.11]{Macaulay-Cesnavi})
    we have $A_2'$ is finite over $A_2$.
    This gives a finite birational ring map $A\to A_1\times A'_2$.
    Let $\fc(A)$ be the conductor of this map.

    Let $U$ (resp. $U_1$) be the punctured spectrum of $A$ (resp. $A_1$).
    Then we have $U=U_1\sqcup U_2$.
    Therefore $\fc(A)$ is either $\fm_A$-primary or $A$.
    As $U$ is $(S_2)$ we have $U_2$ is $(S_2)$,
so $A_2'$ is $(S_2)$ by \cite[Th\'eor\`eme 5.10.5]{EGA4_2}.
In particular $A\neq A_1\times A_2'$
as $A$ is local and not $(S_2)$,
so $\fc(A)$ is $\fm_A$-primary.

A maximal ideal (resp. minimal prime) of $A_2'$ lies above $\fm_A$ (resp. the radical of an element in $\Sigma_2$),
as $A\to A_2'$ is finite (resp. there exists an element $f\in A_2$ that is a nonzerodivisor on both $A_2$ and $A_2'$ so that $(A_2)_f=(A_2')_f$).
Therefore \citestacks{02IJ} tells us for all $\fM'\in\Max(A_2')$ and $\fP'_0\in\Min(A_2')$ with $\fM'\supseteq\fP'_0$,
we have $\HT(\fM'/\fP'_0)>1$,
in particular $\HT(\fM')>1$.
\end{Discu}
\begin{Lem}\label{lem:S2assignment}
    Let $\bQ$=``$(S_1)$'' and $\bP$=``$(S_2)$.''
    Then $\fc(-)$ as in Discussion \ref{discu:cofS2} is a strictly functorial $\bP$-assignment
    on $\cA_\bP^\bQ$.
\end{Lem}
\begin{proof}
Again, as an Artinian ring is $(S_2)$ the $\fm_A$-primary ideal $\fc(A)$ is nonzero.
    It remains to show
    for $\varphi:A\to B$ in $\cA^\bQ_\bP$,
    we have $\fc(A)B=\fc(B)$.

    We know $\varphi$ has $(S_1)$ fibers and its closed fiber has dimension $0$.
    For $Q\in\Sigma_1$,
    $B/QB$ is therefore $1$-dimensional and $(S_1)$,
    so all prime divisors $\fP$ of $QB$ are such that $\dim(B/\fP)=1$.
    If we can show for all $Q\in\Sigma_2$ and all prime divisors $\fP$ of $QB$ (which are automatically minimal),
    we have $\dim(B/\fP)>1$,
    then it will follow that $A_1\otimes_A B=B_1$ and 
    $A_2\otimes_A B=B_2$,
    so $A_2'\otimes_A B=B_2'$
    and $\fc(A)B=\fc(B)$.
    
    We have a commutative diagram
    \[
    \begin{CD}
        A_2@>>> A_2'\\
        @VVV @VVV\\
       A_2\otimes_A B@>>> A_2'\otimes_A B
    \end{CD}
    \]
    of rings.
    The ring $A_2'\otimes_A B$ is $(S_2)$ (as $A_2'$ and the fibers of $\varphi$ are) and universally catenary, hence locally equidimensional \cite[Corollaire 5.10.9]{EGA4_2}.
    Let $\fN'\in\Max(A_2'\otimes_A B)$.
    Then $\fN'\cap (A_2\otimes_A B)$ is the maximal ideal $\fN$ of the local ring $A_2\otimes_A B$,
    so $\fN'\cap A_2$ is the maximal ideal of $A_2$,
    so $\fN'\cap A_2'\in\Max(A_2')$.
    By flatness $\HT(\fN')\geq \HT(\fN'\cap A_2')>1$.
    As $\fN'$ was arbitrary, a similar discussion as the case of $A_2$ tells us for
    all $\fQ_0\in\Min(A_2\otimes_A B)$
    we have $\dim((A_2\otimes_A B)/\fQ_0)>1$,
    as desired.
\end{proof}

\begin{Discu}\label{discu:dualcplxofS2}
    Let $A$ be an $(S_2)$ Noetherian local ring that admits a normalized dualizing complex $\omega$.
    Then $A$ is catenary \citestacks{0A80} 
    and  $(S_2)$,
    so $A$ is equidimensional \cite[Corollaire 5.10.9]{EGA4_2}.
    We will use the standard facts \citetwostacks{0A7U}{0A7V} of dualizing complexes without explicit reference.
    
    We know $\omega\in D^{[-d,-p]}(A)$, where $d=\dim A, p=\depth A$, $H^{-p}(\omega)\neq 0$,
    and $\Supp(H^{-d}(\omega))=\Spec(A)$, as $A$ is equidimensional.
\end{Discu}

\begin{Discu}\label{discu:cofSk}
   Let $\bQ$=``$(S_2)$'' and $\bP$=``$(S_k)$,'' where $k\geq 3$.
   Let $A\in {^d\cA^\bQ_\bP}$,
   and let $p=\depth A$.
   
Let $\omega$ be a normalized dualizing complex of the complete local ring $A$.
Let $\fp\in\Spec(A)$ be of height $d-1$.
Then $\omega_\fp\in D^{\geq -d}(A_\fp)$ and $H^{-d}(\omega_\fp)\neq 0$.
This tells us for all $b>-1-\depth A_\fp$,
we have $H^b(\omega_\fp)=0$.
As $A_\fp$ is $(S_k)$ we have $\depth A_\fp\geq\min\{d-1,k\}$,
so for all $b>-1-\min\{d-1,k\}$,
$H^b(\omega_\fp)=0$.
As $A$ is not $(S_k)$,
$p<\min\{d,k\}$,
so $-p>-\min\{d,k\}\geq -1-\min\{d-1,k\}$.
This tells us $H^{-p}(\omega_\fp)=0$,
in other words, the support of the nonzero module $H^{-p}(\omega)$ is $\{\fm_A\}$.
By local and Matlis duality \citetwostacks{0A84}{08Z9} we see $H^p_{\fm_A}(A)$ is nonzero and of finite length.
   \end{Discu}

   \begin{Lem}\label{lem:Skassignment}
           Let $\bQ$=``$(S_2)$'' and  $\bP$=``$(S_k)$'' $(k\geq 3)$.
    Then $\fc(A)=\Ann_A(H^{\depth A}_{\fm_A}(A))$
    is a strictly functorial $\bP$-assignment on $\cA^\bQ_\bP$.
   \end{Lem}
   \begin{proof}
       By Discussion \ref{discu:cofSk} 
       $H^{\depth A}_{\fm_A}(A)$ is of finite length,
       so 
       $\fc(A)$ is $\fm_A$-primary.
       Again, as an Artinian ring is $(S_k)$ we see $\fc(A)\neq 0$.
       For all $\varphi:A\to B$ in $\cA^\bQ_\bP$,
       the closed fiber of $\varphi$ has dimension $0$,
       so we have $\depth A=\depth B$ \citestacks{0337}.
       This shows $\fc(A)B=\fc(B)$.
   \end{proof}

\begin{Rem}\label{rem:CMassignment}
    It follows formally that for $\bQ$=``$(S_2)$'' and  $\bP$=``Cohen--Macaulay'' we have a strictly functorial $\bP$-assignment on $\cA^\bQ_\bP$.
    Indeed, $^d\cA^\bQ_\bP=\emptyset$ for $d\leq 2$,
    and for $d>2$ and $A\in {^d\cA^\bQ_\bP}$ we let $\fc(A)$  be as in Lemma \ref{lem:Skassignment} for $k=d$.
    It also happens that for all $d$,
    the formula for $\fc(A)$ is the same,
    $\fc(A)=\Ann_A(H^{\depth A}_{\fm_A}(A))$.
\end{Rem}

Basics about Fitting ideals of a finite module can be found in \citestacks{07Z6} and \cite[Chapter 20]{Eisenbud-CA}.
The \emph{Fitting invariant} of a finite module $M$ over a Noetherian ring $A$ 
is the first nonzero Fitting ideal of $M$.
$M$ is projective of constant rank  if and only if the Fitting invariant of $M$ is $A$,
see \citestacks{07ZD}.

\begin{Discu}\label{discu:cofGor}
   Let $\bQ$=``$(S_2)$'' and $\bP$=``Gorenstein.''
   Let $A\in {^d\cA^\bQ_\bP}$.

Let $\omega$ be a normalized dualizing complex of the complete local ring $A$.
   When $d=0$, we let $\fc(A)$ be the Fitting invariant of the module $H^0(\omega)$,
   which is nonzero by definition,
   and is $\fm_A$-primary as $\dim A=0$ and as $\omega$ is not free.
   When $d>0$,
   let $\fc(A)=\Fit_1(H^{-d}(\omega))\cap\Ann_A(H^{1-d}(\omega))\cap\ldots\cap\Ann_A(H^{0}(\omega))$.
   As the punctured spectrum of $A$ is Gorenstein and as $\Supp(H^{-d}(\omega))=\Spec(A)$ (Discussion \ref{discu:dualcplxofS2}),
   we see $\fc(A)$ is $\fm_A$-primary,
   and therefore nonzero as $\dim A>0$.
\end{Discu}

\begin{Lem}\label{lem:Gorassignment}
       Let $\bQ$=``$(S_2)$'' and $\bP$=``Gorenstein.''
       Then $\fc(-)$ as in Discussion \ref{discu:cofGor} is a strictly functorial $\bP$-assignment on $\cA^\bQ_\bP$.
\end{Lem}
\begin{proof}
We have seen $\fc(A)$ is $\fm_A$-primary and nonzero.
    Strict functoriality follows immediately from the fact Fitting ideals commute with base change
    \citestacks{07ZA}
    and that for $\varphi:A\to B$ in $\cA^\bQ_\bP$
    and a normalized dualizing complex $\omega$ of $A$,
    $\omega\otimes^L_A B$ is a normalized dualizing complex of $B$.
    See for example \cite[Lemma 7.1]{Lyu-dual-complex-lift}, note $\dim A=\dim B$.
\end{proof}

\section{A lci-assignment}\label{sec:lciAssign}
 Let $\bQ$=``$(S_2)$'' and $\bP$=``lci.''
 Let $A\in \cA^\bQ_\bP$.
Intuitively, we want to define $\fc(A)$ to be the Fitting invariant of modules $C_n(A/R)$ showing up in \cite{Briggs-Iyenger-Cotangent-Complex},
where $R$ is a regular local ring mapping surjectively to $A$.
The flatness of $C_n(A/R)$ characterizes lci.
However, these modules depend on the choice of $R$ and a projective resolution (in a way that does not change the Fitting invariant, however), 
and are fragile along ascent (\emph{i.e.} still involve non-finite modules).
We will work with $C_n(A/\bZ)$ instead,
which gives the same Fitting invariant.
We use standard notations for derived categories, 
and cohomological conventions for cotangent complexes,
as in \cite{stacks}.

\begin{Discu}\label{discu:modulesC}
    Let $A$ be a ring, and let $L\in D^{-}(A)$.
    For every bounded above complex of projectives $P^\bullet$ that represents $L$,
    we consider the module $C^a(P^\bullet)=H^a(\sigma_{\leq a}P^\bullet)$,
    where $\sigma_{\leq a}$ is the stupid truncation \citestacks{0118}.
    In other words, $C^a(P^\bullet)$ is the cokernel of the map $P^{a-1}\to P^a$,
    which is the module appearing at degree $a$ in $\tau_{\geq a}(P^\bullet)$.
    There is an obvious compatibility with shift and base change.

    The collection of all such $C^a(P^\bullet)$ is denoted $\cC^a(L)$.
    For $X,Y\in\cC^a(L)$,
    there exist projective modules $P,Q$ so that $X\oplus P\cong Y\oplus Q$, see \cite[(7.2)]{Briggs-Iyenger-Cotangent-Complex}.
    We write $C^a(L)$ for an unspecified element in $\cC^a(L)$.
    The flat and projective dimensions of $C^a(L)$ are well-defined.

    If $L$ has tor-amplitude in $[a,b]$,
    then $C^a(L)$ is flat.
    Indeed, let $P^\bullet$ represent $L$,
    then $\tau_{\geq a}(P^\bullet\otimes_A M)$ represents $L\otimes_A^L M$ for all $A$-modules $M$,
    as $L\otimes_A^L M\in D^{[a,b]}(A)$.
    Unwinding the definitions,
    we see $C^a(P^\bullet)\otimes_R M=C^a(P^\bullet)\otimes^L_R M$,
    as desired.
    This also tells us
    if $C^a(L)$ has projective dimension $p<\infty$,
    then $L$ has projective-amplitude in $[a-p,b]$.

    If $L$ has projective-amplitude in $[a,b]$,
    then $C^a(L)$ is projective.
    This is because we can take $P^\bullet$ with $P^{m}=0$ for $m<a$,
    so $C^a(P^\bullet)=P^a$.
    \end{Discu}
    \begin{Discu}\label{discu:TrianglesandC}
    Let $L'\to L\to L''\to +1$ be a distinguished triangle.
    Given representations $P'^\bullet$ of $L'$ and  $P''^\bullet$ of $L''$,
    we can find a representation $P^\bullet$ of $L$ so that the triangle is realized by a short exact sequence of complexes $P'^\bullet\hookrightarrow P^\bullet\twoheadrightarrow P''^\bullet$.
    Indeed, $P^\bullet$ is the cone of any map $P''^\bullet\to P^{\bullet}[1]$ representing $L''\to L[1]$.
    Truncating, we get an exact sequence
    \[
    \begin{CD}
        H^{a-1}(L'')@>>> C^a(P'^\bullet)@>>> C^a(P^\bullet)@>>> C^a(P''^\bullet)@>>> 0.
    \end{CD}
    \]
\end{Discu}

\begin{Lem}\label{lem:StructureofCotangentModule}
    Let $R$ be a Noetherian lci ring and let $A$ be a finitely generated $R$-algebra.
    Let $a\in\bZ,a<-\dim A-1$.
    Then there exist projective modules
    $P,Q$ and a finite module $M$ so that
    $C^a(L_{A/\bZ})\oplus P\cong M\oplus Q$.
\end{Lem}
\begin{proof}
    Apply Discussion \ref{discu:TrianglesandC} to the triangle 
    \[
    \begin{CD}
        L_{A/\bZ}@>>> L_{A/R}@>>> (L_{R/\bZ}\otimes^L_R A)[1]@>>> +1,
    \end{CD}
    \]
    we get an exact sequence
    \[
    \begin{CD}
        H@>>> C^a(L_{A/\bZ})@>>> C^a(L_{A/R})@>>> P@>>> 0.
    \end{CD}
    \]
    where $P=C^{a+1}(L_{R/\bZ}\otimes^L_R A)$,
    $H=H^{a}(L_{R/\bZ}\otimes^L_R A)$.
    Since $R$ is lci, $L_{R/\bZ}$ has tor-amplitude in $[-1,0]$ \cite[(1.2) and (5.1)]{avramov-ci},
    so $H=0$.
    Moreover, every flat $A$-module  has projective dimension $\leq \dim A$ \cite[Seconde partie, Corollaire 3.2.7]{Raynaud}.
    Therefore $L_{R/\bZ}$ has projective-amplitude in $[-\dim A-1,0]$,
    so $P$ is projective,
    and we get $C^a(L_{A/\bZ})\oplus P\cong C^a(L_{A/R})$.
    It remains to observe  $C^a(L_{A/R})$ is finite up to projective summands,
    as $L_{A/R}\in D_{Coh}(A)$
    \citestacks{08PZ}.
\end{proof}
\begin{Rem}\label{rem:ImproveD+1to2}
    If $A$ is countable,
    then we can improve $-\dim A-1$ to $-2$,
    see \cite[Seconde partie, Corollaire 3.3.2]{Raynaud}.
    We could, if necessary, work extensively with countable rings, 
    via a L\"owenheim--Skolem type argument, cf. \cite{Lyu-elementary-subring}.
\end{Rem}

We would like to define the Fitting invariant of $C^a(L_{A/\bZ})$ to be that of $M$;
we will show this is well-defined.
Before that, note the following variant of the main theorem of \cite{Briggs-Iyenger-Cotangent-Complex}.
\begin{Thm}\label{thm:lciandCotangentModule}
    Let $R$ be a Noetherian lci ring and let $A$ be a finitely generated $R$-algebra of finite tor dimension as an $R$-module.
    Let $a\in\bZ,a<-1$.
    Then $A$ is lci if and only if $C^a(L_{A/\bZ})$ is flat.
\end{Thm}
\begin{proof}
    Take the same exact sequence as in the proof of Lemma \ref{lem:StructureofCotangentModule}.
    We have $H=0$ and $P$ is flat.
    Thus $C^a(L_{A/\bZ})$ is flat if and only $C^a(L_{A/R})$ is flat,
    if and only if $R\to A$ is lci (\cite[Theorem B]{Briggs-Iyenger-Cotangent-Complex} and \cite[(1.2)]{avramov-ci}),
    if and only if $A$ is lci \cite[(5.4) and (5.9)]{avramov-ci}.
\end{proof}

\begin{Def}\label{def:Fitting}
    Let $A$ be a Noetherian local ring.
    We say an $A$-module $X$ is \emph{finite-by-flat} if there exists a finite submodule $M$ of $X$ such that $X/M$ is flat.
The \emph{Fitting invariant of $X$} is defined to be the Fitting invariant of $M$.
    
    We say $X$ is \emph{pseudo-finite-by-flat}
    if there exists a flat $A$-module $C$ so that $X\oplus C$ is finite-by-flat.
    The \emph{Fitting invariant of $X$} is defined to be the Fitting invariant of $X\oplus C$.

    As soon as these invariants are well-defined,
    they are clearly compatible with each other and the Fitting invariant of finite modules,
    as a finite flat module is free
    and as taking a direct sum with a finite free module does not change the Fitting invariant \citestacks{07ZA}.
    It is also clear that the Fitting invariant of $X$ is $A$ if and only if $X$ is flat.
 \end{Def}
\begin{Lem}\label{lem:FitofFinitebyFlat}
Let $A$ be a Noetherian local ring.
Then the following hold.
\begin{enumerate}[label=$(\roman*)$]
    \item\label{Fit:wdf} The Fitting invariant of a finite-by-flat or a pseudo-finite-by-flat module is well-defined.
    \item\label{Fit:cokerFlatSameFit} Given an inclusion of modules $X\subseteq Y$ with flat quotient,
    if $X$ is finite-by-flat (resp. pseudo-finite-by-flat), so is $Y$,
    and the Fitting invariants of $X$ and $Y$ are the same.
\end{enumerate}
\end{Lem}
\begin{proof}
    We first show \ref{Fit:wdf} for finite-by-flat modules.
    Let $M,M'$ be two finite submodules of a given module $X$ so that $X/M$ and $X/M'$ are flat.
    By Lazard's Theorem \citestacks{058G}, we can write
    $X/M=\colim_\alpha L_\alpha$,
    where the colimit is filtered and $L_\alpha$ are finite free.
    Let $X_\alpha=X\times_{X/M}L_\alpha$,
    so we have a commutative diagram with exact rows
    \[
    \begin{CD}
        0@>>> M@>>> X_\alpha@>>> L_\alpha@>>> 0\\
        @. @| @VVV @VVV @.\\
         0@>>> M@>>> X@>>> X/M@>>> 0.
    \end{CD}
    \]
    We have $X=\colim_\alpha X_\alpha$ as filtered colimits 
    commute
with finite limits.
    Therefore the inclusion $M'\subseteq X$ factors through some $X_\alpha$ 
\citestacks{0G8P}.
    By \citestacks{058M} $M'\subseteq X$ is pure (\emph{i.e.} universally injective),
    so $M'\to X_\alpha$ is also pure,
    thus split \citestacks{058L}.
    On the other hand $X_\alpha\cong M\oplus L_\alpha$ as $A$-modules as $L_\alpha$ is free.
    We conclude that $M'$ is isomorphic to a direct summand of $M\oplus F$ where $F$ is a finite free $A$-module,
    and by symmetry 
    $M$ is isomorphic to a direct summand of $M'\oplus F'$ where $F'$ is a finite free $A$-module.
    If $A$ is complete,
    it follows from Krull--Schmidt \cite[Corollary 1.10]{Leuschke-Wiegand-Cohen-Macaulay-Modules}
that    $M\oplus P\cong M'\oplus P'$ for some finite free $A$-modules $P,P'$.
For a general $A$ the same is true by \cite[Corollary 1.15]{Leuschke-Wiegand-Cohen-Macaulay-Modules}.
Therefore the Fitting invariants of $M$ and $M'$ are the same  \citestacks{07ZA}.

Given an inclusion of modules $X\subseteq Y$ with flat quotient,
if $M$ is a finite submodule of $X$ so that $X/M$ is flat,
then $M$ is a finite submodule of $Y$ so that $Y/M$ is flat,
as $Y/M$ is an extension of $X/M$ by $Y/X$.
This gives \ref{Fit:cokerFlatSameFit} in the finite-by-flat case.

Next, let $X$ be a pseudo-finite-by-flat module and $C,D$ be flat modules so that $X\oplus C$
 and $X\oplus D$ are both finite-by-flat.
 Then $X\oplus C\oplus D$ is finite-by-flat and has the same Fitting invariant as $X\oplus C$
 and $X\oplus D$
 by \ref{Fit:cokerFlatSameFit} for finite-by-flat modules.
Therefore the Fitting invariants of $X\oplus C$
 and $X\oplus D$ are the same,
 which is \ref{Fit:wdf} for $X$.

Finally, given an inclusion of modules $X\subseteq Y$ with flat quotient,
if $C$ is a flat module, then we have an inclusion of modules $X\oplus C\subseteq Y\oplus C$ with flat quotient.
This gives \ref{Fit:cokerFlatSameFit} in the pseudo-finite-by-flat case. 
 \end{proof}
To summarize, Lemma \ref{lem:StructureofCotangentModule}, Theorem \ref{thm:lciandCotangentModule},
and Lemma \ref{lem:FitofFinitebyFlat} give
\begin{Thm}\label{thm:FittofLAoverZ}
Let $R\to A$ be a finite type ring map
where $R$ is Noetherian and lci and $A$ is local. 
Let $a\in \bZ,a<-\dim A-1$.
Then the following hold.
\begin{enumerate}[label=$(\roman*)$]
    \item The module $C^a(L_{A/\bZ})$ as in Discussion \ref{discu:modulesC} is pseudo-finite-by-flat.
    \item If $A$ is of finite tor dimension as an $R$-module,
    then the Fitting invariant of $C^a(L_{A/\bZ})$ is $A$ if and only if $A$ is lci.
\end{enumerate} 
\end{Thm}
\begin{Rem}
    In fact, Theorem \ref{thm:FittofLAoverZ} holds for all $a<-2$ (cf. Remark \ref{rem:ImproveD+1to2}).
    This is because $R$ can be approximated by countable subrings using \cite{Lyu-elementary-subring} (and \cite[Corollary 3.4]{Gor-2-rings}).
\end{Rem}

We arrive at our lci-assignment.
\begin{Thm}\label{thm:lci-assignment}
    Let  $\bQ$=``$(S_2)$'' and $\bP$=``lci.''
    Let $A\in {^d\cA^\bQ_\bP}$.
    Let $a=-d-2$.
    Let $\fc(A)$ be the Fitting invariant of $C^a(L_{A/\bZ})$.
    Then $\fc(-)$ is a strictly functorial $\bP$-assignment on $\cA^\bQ_\bP$.
\end{Thm}
\begin{proof}
We can find a complete regular local ring $R$ and a surjective ring map $R\to A$. 
    By Theorem \ref{thm:FittofLAoverZ},
    $\fc(A)$ is well-defined and $0\neq \fc(A)\neq A$.
    To show $\fc(-)$ is strictly functorial,
    let $\varphi:A\to B$ be in $ {^d\cA^\bQ_\bP}$.
    Apply Discussion \ref{discu:TrianglesandC} to 
    \[\begin{CD}
        L_{A/\bZ}\otimes^L_A B@>>> 
 L_{B/\bZ}@>>> L_{B/A}@>>> +1,
    \end{CD}\]
    we get 
    an exact sequence
    \[
    \begin{CD}
        H@>>> C^a(L_{A/\bZ})\otimes_A B@>>> C^a(L_{B/\bZ})@>>> C@>>> 0
    \end{CD}
    \]
    where $H=H^{a-1}(L_{B/A})$ and $C=C^a(L_{B/A})$.
    As $\varphi$ is lci and as $a\leq -1$ we have $H=0$ and $C$ flat (in fact projective since $a=-d-2$, cf. proof of Lemma \ref{lem:StructureofCotangentModule}),
    so Lemma \ref{lem:FitofFinitebyFlat}\ref{Fit:cokerFlatSameFit} gives $\fc(A)B=\fc(B)$.

    It remains to show $\fc(A)$ is $\fm_A$-primary.
    When $d=0$ this is trivial, so we assume $d>0$.
    Let $P^\bullet$ be a complex of finite free modules that satisfies $P^{>-1}=0$ and represents $L_{A/R}$
    \citetwostacks{08PZ}{08QF}.
    As seen in Lemma \ref{lem:StructureofCotangentModule}
    $\fc(A)$ is the Fitting invariant of $C^a(P^\bullet)$.
    We will show $M:=C^a(P^\bullet)$ is finite flat of constant rank, say $r$, on the punctured spectrum of $A$.
    Then by \citestacks{07ZD},
    for $\fa=\Fit_j(M)\ (j<r)$,
    we have $\fa_\fp=0$ for all $\fp\in\Spec(A)\setminus \{\fm_A\}$,
    so $\fa=0$
as $\depth A\geq 1$;
therefore $\fc(A)=\Fit_r(M)$, and $\fc(A)_\fp=A_\fp$.

    We know $M$ is finite flat on the punctured spectrum of $A$ which is lci.
    If $d\geq 2$, then $\depth A\geq 2$,
    as $A$ is $(S_2)$.
    Therefore the punctured spectrum of $A$ is connected \citestacks{0BLR},
    so the rank is constant.
    We may now assume $d=1$.

Let $I=\ker(R\to A)$, $\fp\in V(I)\setminus\Max(R)$.
    The complex $(\tau_{\geq a} P^\bullet)_\fp$ represents $(I/I^2)_\fp[1]$ \citestacks{08SJ},
    as $A_\fp$ is lci.
    Computing Euler characteristic,
    we see
    $(-1)^a\rank M_\fp + \sum_{i=a+1}^{-1} (-1)^i\rank P^i=-\rank (I/I^2)_\fp$.
    We know $\rank (I/I^2)_\fp=\dim R_\fp-\dim A_\fp$,
    as $I_\fp$ is generated by a regular sequence.
    As $d=1$ and as $R$ is a catenary domain, we have $\dim R_\fp=\dim R-1,\dim A_\fp=0$,
    independent of the choice of $\fp$.
    Therefore,
    $\rank M_\fp$ is independent of the choice of $\fp$,
    as desired.
\end{proof}

\section{Local lifting}\label{sec:Local-lifting-result}

\begin{Thm}\label{thm:local-lift-P}
    Let $\bP$ be the property ``$(S_k)$'' $(k\geq 0)$, ``Cohen--Macaulay,'' ``Gorenstein,'' or ``lci.''
    
    Let $R$ be a Noetherian semilocal ring, $I$ an ideal of $R$.
    Assume 
    \begin{enumerate}
        \item $R$ is $I$-adically complete.
        \item\label{locallift:quotPring} $R/I$ is a $\bP$-ring.
    \end{enumerate}
    Then $R$ is a $\bP$-ring.
\end{Thm}
\begin{proof}
    First consider the case $\bP$=``$(S_1)$,''
    as every Noetherian ring is $(S_0)$.
    Let $\bQ$ be the trivial property.
    A strictly functorial $\bP$-assignment on $\cD_\bP$ exists,
    Lemmas \ref{lem:extendPassignment} and \ref{lem:S1assignment}.
    The assumptions \eqref{Nsmr:Qring} and \eqref{Nsmr:Q-ify} in Theorem \ref{thm:Nishimura-local-lifting-argument} are trivial,
    and we conclude.

    Next, consider the case $\bP$=``$(S_2)$.''
    Let $\bQ$=``$(S_1)$.''
    A strictly functorial $\bP$-assignment on $\cD^\bQ_\bP$ exists,
    Lemmas \ref{lem:extendPassignment} and \ref{lem:S2assignment}.
    The assumption 
    \eqref{Nsmr:Q-ify} in Theorem \ref{thm:Nishimura-local-lifting-argument} is trivial as a domain in $(S_1)$,
    whereas \eqref{Nsmr:Qring} follows from the case $\bP$=``$(S_1)$,''
    and we conclude.

    Finally, consider the case $\bP$=``$(S_k)$'' $(k\geq 3)$, ``Cohen--Macaulay,'' ``Gorenstein,'' or ``lci.''
    Let $\bQ$=``$(S_2)$.''
    A strictly functorial $\bP$-assignment on $\cD^\bQ_\bP$ exists,
    Lemma \ref{lem:extendPassignment}
    and Lemma \ref{lem:Skassignment},
    Remark \ref{rem:CMassignment},
    Lemma \ref{lem:Gorassignment},
    and Theorem \ref{thm:lci-assignment}.
    The assumption \eqref{Nsmr:Qring} in Theorem \ref{thm:Nishimura-local-lifting-argument} follows from the case $\bP$=``$(S_2)$,''
    and    \eqref{Nsmr:Q-ify}  follows from Theorem \ref{thm:semi-Nagatacharacterizelocal} (or \cite[Corollary 2.14]{Macaulay-Cesnavi}),
    and we conclude.
\end{proof}
\begin{Rem}\label{rem:PisGeometric}
    All properties $\bP$ in Theorem \ref{thm:local-lift-P}
    are preserved by finite field extensions,
    so being a $\bP$-ring is the same as having $\bP$ formal fibers.
    This is because a finite extension of fields is a syntomic ring map, cf. \citestacks{00SK}. 
\end{Rem}

From the case $\bP$=Cohen--Macaulay
    and the same argument as in \cite[\S 8]{Lyu-dual-complex-lift} we get
\begin{Cor}\label{cor:local-lift-quot-CM}
Let $R$ be a Noetherian semilocal ring, $I$ an ideal of $R$.
    Assume 
    \begin{enumerate}
        \item $R$ is $I$-adically complete.
        \item\label{quotCM} $R/I$ is a quotient of a Cohen--Macaulay ring.
    \end{enumerate}
    Then $R$ is a quotient of a Cohen--Macaulay ring.
\end{Cor}

\section{Local lifting of properties in codimension zero and Cohen--Macaulayness in codimension one}\label{sec:Local-lifting-P0-CM1}

In this section we exploit the lifting problem for $\bP_k$=``$\bP$ in codimension $k$.''  
One important phenomenon is that Grothendieck localization fails (Example \ref{Exam:Gro-localize-false-R0}),
which forces us to add a universally catenary condition in Theorems \ref{thm:Nishimura-local-lifting-argument-P0} and \ref{thm:Nishimura-local-lifting-argument-CM1}.
Another new difficulty is that $\bP_k$ does not imply $(S_1)$ even if $\bP$ does,
so we are not able to apply Lemma \ref{lem:extendPassignment},
hence it is difficult (and maybe impossible) to find assignments $\fc(-)$ as formulated in Definition \ref{def:CategoryDQP}.
Even if we have a candidate $\fc(-)$, we also run into the problem that $\fc(A)$ can be zero when $U_\bP(A)=\emptyset$, see Discussion \ref{discu:P0Assign},
which forces us to add assumption \eqref{NsmrP0:P0-ify-codim-1} in Theorem \ref{thm:Nishimura-local-lifting-argument-P0}.

\begin{Discu}\label{discu:properties-k}
    Let $\bP$ be a property of Noetherian rings.
    Let $k\in \bZ_{\geq 0}$.
    We denote by $\bP_k$ the property of ``$\bP$ in codimension $k$,''
    that is, a Noetherian ring $A$ satisfies $\bP_k$ if and only if $A_\fp$ satisfies $\bP$ for all $\fp\in\Spec(A), \HT(\fp)\leq k$.
    In particular, Serre's property $(R_k)$ is $\bP_k$ for $\bP$=``regular.'' 

    It is trivial that if $\bP$ satisfies the condition \ref{PisSing} or \ref{PisPointwise} in Discussion \ref{discu:Properties}, so does $\bP_k$ for all $k$. 
    It is also clear that if $\bP$ satisfies the condition \ref{Pascends} or \ref{Pdescends} in Discussion \ref{discu:Properties}, so does $\bP_k$ for all $k$,
    cf. \citestacks{00ON}. 

    Let $A$ be a Noetherian ring such that $U_\bP(A)$ is open, say equal to the complement of $V(\fa)$ in $\Spec(A)$,
    where $\fa$ is a radical ideal.
    Let $\fp_1,\ldots,\fp_m$ be all the prime divisors of $\fa$ whose height is at most $k$.
    Then it is clear that $U_{\bP_k}(A)$ is open and equal to the complement of $V(\fp_1\cap\ldots\cap\fp_m)$ in $\Spec(A)$.
    Therefore, if $\bP$ satisfies the condition \ref{PisOpen} in Discussion \ref{discu:Properties}, so does $\bP_k$ for all $k$. 
    
    Similarly, for any $\bP$ and any $A$, $U_{\bP_0}(A)$ is always open, and its complement is the union of all $V(\fp)$, where $\fp\in\Min(A)$ is such that $A_\fp$ does not satisfy $\bP$.
\end{Discu}

In general, one cannot expect $\bP_k$ to satisfy Grothendieck localization (\ref{PGroLocalizes} in Discussion \ref{discu:Properties}).
\begin{Exam}\label{Exam:Gro-localize-false-R0}
    Let $A=k[y]_{(y)},B=A[x_1,\ldots,x_m,z]_{(x_1,\ldots,x_m,y,z)}$ and $C=B/((x_1^2,\ldots,x_m^2,y+z)\cap (z))$.
    Note that the sequences $z,y$; $x_1^2,\ldots,x_m^2,y+z,y$; and $x_1^2,\ldots,x_m^2,y+z,z$ are all regular sequences in $B$,
    hence $A\to C$ is flat, and $C=B/(zx_1^2,\ldots,zx_m^2,z(y+z))$.
    The closed fiber of the flat local ring map $A\to C$ is then (the spectrum of) $k[x_1,\ldots,x_m,z]_{(x_1,\ldots,x_m,z)}/(zx_1^2,\ldots,zx_m^2,z^2)$, which is geometrically $(R_k)$ for all $k<m$ but not $(R_m)$.
    On the other hand, the generic fiber of $A\to C$ is not even $(R_0)$,
    as the minimal prime $(x_1,\ldots,x_m,y+z)$ of $C$ is not in the regular locus of $C$.
\end{Exam}

On the positive side, we have the following, cf. \cite[Th\'eor\`eme 1.2]{ionescu-Rk-lifting}.
\begin{Thm}\label{thm:RkGroLocalizes}
    Let $\bP$ be a property of Noetherian rings.
    Let $k\in \bZ_{\geq 0}$.
    Assume that $\bP$ satisfies 
    \ref{PisPointwise}\ref{PGroLocalizes} in Discussion \ref{discu:Properties}.

    Let $\varphi:A\to B$ be a flat local map of Noetherian local rings.
    Assume
    \begin{enumerate}
        \item\label{RkGro:AisRkRing} $A$ is a $\bP_k$-ring.
        \item\label{RkGro:closedFiber} The closed fiber of $\varphi$ is geometrically $\bP_k$.
        \item\label{RkGro:UCeqd} $B/\fp B$ is catenary and equidimensional for all $\fp\in\Spec(A)$.
    \end{enumerate}
        Then $\varphi$ is a $\bP_k$-map.
\end{Thm}
\begin{Rem}\label{Rem:UCeqd-BimpliesBoverpB}
   If $B$ is catenary and equidimensional, then \eqref{RkGro:UCeqd} is true, see \citestacks{0AW4}.
\end{Rem}
\begin{proof}
By Noetherian induction, we may assume $A$ is an integral domain and that $A/I\to B/IB$ is a $\bP_k$-map for all nonzero ideals $I\subseteq A$. 
Let $K$ be the fraction field of $A$.
We need to show $B\otimes_A K$ is a geometrically $\bP_k$ $K$-algebra.
Let $\fq\in\Spec(B)$ be above $0\in\Spec(A)$ and of height $\leq k$.
By \ref{PisPointwise}, it suffices to show $B_\fq$ is a geometrically $\bP$ $K$-algebra.


Let $\underline{x}$ be a system of parameters of $A$, and let $\fQ$ be a minimal divisor of $\fq+(\underline{x})B$.
Then $\HT(\fQ/\fq)\leq \dim A$, so $\HT(\fQ)\leq \HT(\fq)+\dim A$ as $B$ is catenary and equidimensional.
By \citestacks{00ON} we have $\HT(\fQ/\fm B)\leq \HT(\fq)\leq k$, $\fm$ being the maximal ideal of $A$.
Since $B/\fm B$ is geometrically $\bP_k$ by \eqref{RkGro:closedFiber},
we see $B_\fQ/\fm B_\fQ$ is geometrically $\bP$.
By \ref{PGroLocalizes}, $A\to B_\fQ$ is a $\bP$-map,
hence $B_\fQ\otimes_A K$ is geometrically $\bP$ over $K$, thus so is the localization $B_\fq\otimes_A K=B_\fq$ by \ref{PisPointwise}.
\end{proof}

Next, we discuss local lifting of $\bP_0$-rings.
It is easy to find a strictly functorial $\bP_0$-assignment, as follows.
\begin{Discu}\label{discu:P0Assign}
    Let $\bP$ be a property of Noetherian rings such that $\bP_0$ satisfies \ref{PisSing}--\ref{Pascends} in Discussion \ref{discu:Properties} (e.g. if $\bP$ itself does).
    For a Noetherian ring $A$, let $0=\fq_1\cap\ldots\cap\fq_r\cap\fq_{r+1}\cap\ldots\cap\fq_s$ be a shortest primary decomposition, such that for $\fp_i=\sqrt{\fq_i}$,
    we have $\fp_1,\ldots,\fp_r\in U_{\bP_0}(A)$ and $\fp_{r+1},\ldots,\fp_s\not\in U_{\bP_0}(A)$.

    Write $\fa=\fq_1\cap\ldots\cap\fq_r$ and $\fb=\fq_{r+1}\cap\ldots\cap\fq_s$.
    As seen in Discussion \ref{discu:properties-k} we know $U_{\bP_0}(A)$ is open and its complement is $V(\fb)$.
    Therefore $\fa=H^0_{\fb}(A)$ is independent of the choice of the primary decomposition \cite[Proposition 3.13]{Eisenbud-CA}.
    We define $\fc(A)=\operatorname{Ann}_A\fa$.
    It is clear that $V({\fc(A)})=V(\fb)$ as sets.
    Moreover, $\fc(A)$ depends only on $V({\fb})=\Spec(A)\setminus U_{\bP_0}(A)$ and not on $\fb$,
    thus for a $\bP_0$-map $A\to B$,
    we have $\fc(A)B=\fc(B)$.

    However, as opposed to cases considered in \S\ref{sec:SkCMGorAssign} and \S\ref{sec:lciAssign}, when $U_{\bP_0}(A)=\emptyset$, we have $\fc(A)=0$.
\end{Discu}

Combining Theorem \ref{thm:RkGroLocalizes} and Discussion \ref{discu:P0Assign} we have
\begin{Thm}\label{thm:Nishimura-local-lifting-argument-P0}
    Let $\bP$ a property of Noetherian rings so that $\bP$ satisfies \ref{PisSing}--\ref{Pascends} and \ref{PGroLocalizes} in Discussion \ref{discu:Properties}.

    Let $R$ be a Noetherian semilocal ring, $I$ an ideal of $R$.
    Assume 
    \begin{enumerate}
        \item $R$ is $I$-adically complete.
        \item\label{NsmrP0:quotP0ring} $R/I$ is a $\bP_0$-ring.
        \item\label{NsmrP0:UC} $R$ is universally catenary.
        \item\label{NsmrP0:P-ify-codim-1} For every finite $R$-algebra $B$ that is an integral domain of dimension $\geq 2$,
        there exists a finite inclusion of domains $B\subseteq C$
        such that $U_{\bP}(C)\neq\{0\}$ (e.g. if $U_{\bP}(C)$ is open).
    \end{enumerate}
    Then $R$ is a $\bP_0$-ring.
\end{Thm}
\begin{proof}
Note that $\bP$ satisfies \ref{PisSing}--\ref{PisOpen} in Discussion \ref{discu:Properties}, see Discussion \ref{discu:properties-k}.
Similar to the proof of Theorem \ref{thm:Nishimura-local-lifting-argument} we may assume $R$ is an integral domain and the theorem holds for all $R$ of smaller dimension.
If $\dim R\leq 1$ then either $I=0$ or $R$ is complete, so we may assume $\dim R\geq 2$.
By \eqref{NsmrP0:P-ify-codim-1} we may assume
$U_{\bP}(R)\neq\{0\}$.
Let $0\neq \fP\in U_{\bP}(R)$ and let $P\in\Min(\fP R^\wedge)$.
By induction hypothesis $R/\fP$ is a $\bP_0$-ring, so $(R^\wedge)_P/\fP (R^\wedge)_P$ is $\bP$.
Thus $(R^\wedge)_P$ is $\bP$ by \ref{Pascends},
so $U_\bP(R^\wedge)\neq\emptyset$.

Let $C=\fc(R^\wedge)$ where $\fc(-)$ is as in Discussion \ref{discu:P0Assign}.
Notice that $C\neq 0$ as $U_\bP(R^\wedge)\neq\emptyset$.
The rest of the proof is identical to that of Theorem \ref{thm:Nishimura-local-lifting-argument},
with \ref{PGroLocalizes} replaced by Theorem \ref{thm:RkGroLocalizes}.
We only need to show that for every prime ideal $M$ of $T=((R^\wedge)_\fp)^*$, $T_M$ is catenary and equidimensional (cf. Remark \ref{Rem:UCeqd-BimpliesBoverpB}), and we only need to check maximal $M$.

As $R^\wedge$ is a quotient of a regular ring, so are $S:=(R^\wedge)_\fp$ and $T$, thus we only need to check $T_M$ is equidimensional.
As $M$ is maximal the map $S_{M\cap S}\to T_M$ induces an isomorphism of completions.
If $S_{M\cap S}$ is equidimensional, so is its completion \citestacks{0AW3}, thus so is $T_M$ \citestacks{0AW4}.
As $S$ is a localization of $R^\wedge$ it now suffices to show $(R^\wedge)_Q$ is equidimensional for all $Q\in\Spec(R^\wedge)$.
Again, as $R^\wedge$ is catenary it suffices to check maximal $Q$, for which equidimensionality follows from \citestacks{0AW3} as $R$ is a universally catenary domain by assumption \eqref{NsmrP0:UC}.
\end{proof} 

Slightly modifying the arguments we get
\begin{Thm}\label{thm:Nishimura-local-lifting-argument-CM1}
    Let $\bP$=``Cohen--Macaulay.''
    
Let $R$ be a Noetherian semilocal ring, $I$ an ideal of $R$.
    Assume 
    \begin{enumerate}
        \item $R$ is $I$-adically complete.
        \item\label{NsmrCM1:quotCM1ring} $R/I$ is a $\bP_1$-ring.
        \item\label{NsmrCM1:UC} $R$ is universally catenary.
    \end{enumerate}
    Then $R$ is a $\bP_1$-ring.
\end{Thm}
\begin{proof}
Let $A$ be any Noetherian ring with $U_{\bP_1}(A)$ open.
As an Artinian ring is Cohen--Macaulay it is clear that there exists finitely many primes $\fp_1,\ldots,\fp_m\in\Spec_1(A)\cap \Ass(A)$ so that $U_{\bP_1}(A)=\Spec(A)\setminus V(\fp_1\cap\ldots\cap\fp_m)$.
Therefore, we can similarly take $\fc(A)=\Ann H^0_{\fb}(A)$ where $\fb=\fp_1\cap\ldots\cap\fp_m$,
so $V(\fc(A))=\Spec(A)\setminus U_{\bP_1}(A)$ and
and $\fc(A)B=\fc(B)$ for all $\bP_1$-maps $A\to B$.
Note that we automatically have $\fc(A)\neq 0$.

The rest of the proof is now verbatim to that of Theorem \ref{thm:Nishimura-local-lifting-argument},
with \ref{PGroLocalizes} replaced by Theorem \ref{thm:RkGroLocalizes}, and with equidimensionality guaranteed by \eqref{NsmrCM1:UC} as seen in the proof of Theorem \ref{thm:Nishimura-local-lifting-argument-P0}.
\end{proof}
\printbibliography
\end{document}